\documentclass[a4paper,12pt]{article}

\usepackage{amssymb}
\usepackage{amsmath}
\usepackage{amsthm}
\usepackage{amscd}
\usepackage[T1]{fontenc}
\usepackage{eucal,mathrsfs,dsfont,verbatim}%gives nicer set-names. Use \mathds instead of \mathbb
\usepackage{lscape,enumerate}
\usepackage{color}%for comments

\usepackage{hyperref}

\renewcommand{\star}{\circledast}
\newcommand{\eps}{\varepsilon}
\newcommand{\prt}{\partial}

\newcommand{\real}{\mathds{R}}
\newcommand{\rn}{{\mathds{R}^n}}

\newcommand{\Ee}{\mathds E}
\newcommand{\Pp}{\mathds P}
\newcommand{\I}{\mathds 1}

\def\1{1\!\!\hbox{{\rm I}}}

\newcommand{\be}{\begin{equation}}
\newcommand{\ee}{\end{equation}}
\newcommand{\ba}{\begin{aligned}}
\newcommand{\ea}{\end{aligned}}

\usepackage{color}%for comments

\evensidemargin
-1cm

\numberwithin{equation}{section}
\theoremstyle{plain}
\newtheorem{theorem}{Theorem}
\newtheorem{lemma}[theorem]{Lemma}
\newtheorem{proposition}[theorem]{Proposition}

\theoremstyle{definition}
\newtheorem{remark}[theorem]{Remark}
\newtheorem{definition}{Definition}
\newtheorem{example}[theorem]{Example}

\numberwithin{theorem}{section}
\numberwithin{definition}{section}
\begin{document}

\title{Parametrix construction for certain L\'evy-type processes}
%\footnote{\vica \textbf{thanks to grants!}\normal }}

\author{%
    \textsc{Victoria Knopova}%
    \thanks{V.M.\ Glushkov Institute of Cybernetics,
            NAS of Ukraine,
            40, Acad.\ Glushkov Ave.,
            03187, Kiev, Ukraine,
            \texttt{vic\underline{ }knopova@gmx.de}}
    \textrm{\ \ and\ \ }
    \stepcounter{footnote}\stepcounter{footnote}\stepcounter{footnote}
    \stepcounter{footnote}\stepcounter{footnote}%
    \textsc{Alexei Kulik}%
    \thanks{Institute of Mathematics, NAS of Ukraine, 3, Tereshchenkivska str., 01601  Kiev, Ukraine,
    \texttt{kulik@imath.kiev.ua}}
    }

\date{}

\maketitle

\begin{abstract}
    \noindent
    In this paper we show that a non-local operator of certain type extends to the generator of a strong Markov process, admitting  the transition probability density.  For this transition probability density we construct the intrinsic upper and lower bounds, and prove some smoothness properties.
 Some examples are provided.


\end{abstract}

\numberwithin{equation}{section}

\section{Introduction}

Consider the equation:
 \begin{equation}
    \begin{split}
    \tfrac{\partial}{\partial t}f(t,x) =L(x,D)f(t,x), \quad t>0, \quad x\in \real, \label{fund1}
    \end{split}
    \end{equation}
where  the operator  $L(x,D)$ is defined on functions from the space $C_\infty^2(\real)$ of twice continuously differentiable functions  vanishing at infinity as
 \begin{equation}
L(x,D)f(x):=\int_\real \big(f(x+u)-f(x)\big)\mu(x,du), \label{lxd0}
\end{equation}
where the kernel $\mu(x,du)$ is symmetric with respect to  $u$ for each fixed $x\in \real$, and
$$
\int_\real (1\wedge u^2)\mu(x,du)<\infty.
$$
We refer to  \cite{Ja01}--\cite{Ja05}  for the detailed survey on  operators of such a type.

In this paper we   develop a version of the parametrix method, which allows to construct the \emph{candidate } $p_t(x,y)$ for being the fundamental solution to \eqref{fund1}, and
provide the \emph{justification procedure}, which shows that $(L(x,D), C_\infty^2(\real))$ extends to the  $C_\infty(\real)$ --generator of a (strong) Feller seimgroup

\begin{equation}\label{t-semi}
T_t f(x)= \int_\real f(y) p_t(x,y)dy, \quad t>0, \quad x\in \real,
\end{equation}
which in turn is in one-to-one correspondence with a Markov process, whose transition probability density  is $p_t(x,y)$.

Although our parametrix construction relies on the scheme described in \cite{Fr64},  see also \cite{Le07} for the original paper,  in our case the situation is much more complicated.
The first difficulty  is that even if we deal with the generator of a L\'evy process, unlikely to the diffusion or to the symmetric $\alpha$-stable case,  in general   one cannot expect    the  fundamental solution to the respective Cauchy problem to possess the
single-kernel estimate
\begin{equation}\label{bell}
g_t(x,y)\leq C \rho_t f(\rho_t (y-x)),
\end{equation}
unless $\mu(du)$ satisfies addition regularity assumptions; here $\rho: (0,1]\to (0,\infty)$ is some "scaling function", and $f\in L_1(\real)$,  see \cite{KK12a}.
 Instead, we have the upper estimate in the form of the convolution of a single kernel and a finite measure, see \eqref{ptx-der}.  This requires the deep modification of the parametrix method, presented in \cite{Fr64}.

For the justification procedure  we use the method, described in \cite{KK14a}, and further developed in \cite{KK14b}.  This  method, after necessary modifications, works as well in the cases treated in this paper, and relies on the \emph{approximative fundamental solution}  to \eqref{fund1}, in particular, on its differentiability properties.

% It was shown in \cite{Ish94}, \cite{Ish01}, \cite{Pi97a}, \cite{Pi97b}, that the transition probability density of
%a Markov process constructed as the solution  of a L\'evy-driven SDE can be very irregular, and its behaviour intrinsically depends on the structure of the L\'evy measure of the noise.
We assume that
\begin{equation}
L(x,D)= L(D)+\mathfrak{L}(x,D), \label{expl1}
\end{equation}
where
\begin{equation}\label{ld}
L(D) f(x):= \int_\real (f(x+u)-f(x))\mu(du),\quad f\in C_\infty^2(\real),
\end{equation}
and   the operator $\mathfrak{L}(x,D)$ is some lower order perturbation of $L(D)$,  i.e.,
\begin{equation}
\mathfrak{L}(x,D)f(x):=\int_{\real} \big(f(x+u)-f(x)\big) m(x,u)\mu(du), \quad f\in C_\infty^2(\real),\label{Qxd}
\end{equation}
were the function $m(x,u)$ is non-negative and bounded from above by $c(1\wedge |u|^\epsilon)$, $c,\epsilon>0$.
Here  $\mu$ is a  L\'evy measure, i.e.
$\int_\real (1\wedge u^2) \mu(du)<\infty$;
in addition we assume $\mu$ to be symmetric.
 We emphasize that the list of our assumptions on $\mu$ and $m(x,u)$ contains neither
 the conditions  on smoothness of the L\'evy measure $\mu$, corresponding to $L(D)$,  nor
   the  condition that $\mu(du)$ is comparable to an $\alpha$-stable L\'evy measure $c_\alpha |u|^{-1-\alpha}du$. In particular, in Section~\ref{example}
 we provide  a) an example in which  for a discrete L\'evy measure $\mu$ one can construct the fundamental
 solution to \eqref{fund1}, and write the estimates for it in a rather compact form; b) an example in which the
 characteristic exponent $q(\xi)$  related to $\mu$ by
 $$
 q(\xi)= \int_\real (1-\cos (\xi u))\mu(du)
 $$
 has oscillations, but still our method  is applicable.

 The paper is organized as follows.  In Section~\ref{settings} we outline our method and present the main results. %, namely,  Theorems~\ref{t-main1}, \ref{t-main2}, \ref{t-jump3} and \ref{loc-time}.
 Proofs are given in Section~\ref{constr}, \ref{cont} and \ref{just}.   In Section~\ref{example}
 we consider several examples which illustrate our results.

 \begin{comment}In the first example we consider the perturbations of the $\alpha$-stable L\'evy  measure. In the second example we  show that one can obtain rather explicit upper bound even in the case when the L\'evy measure is discrete, but "sufficiently regular". Finally, the third example  shows that our method can be applied even in the situation, when the characteristic exponent oscillates between two power functions.
\end{comment}

\subsection{Notation}

 We write $f\asymp g$ if for some positive constants $c_1$, $c_2$   we have $c_1g(x)\leq f(x)\leq c_2 g(x)$; $f\wedge g:= \min(f,g)$, $f\vee g:= \max(f,g)$; $f_+(x):= f(x)\vee 0$. Through the paper we denote by $c_i$   arbitrary positive constants. The symbols  $*$  and $\star$ denote, respectively, the convolutions
 $$
f* g(t,x,y):= \int_\real f(t,x,z)g(t,z,y)dz, \quad f\star g (t,x,y):= \int_0^t \int_\real f(t-\tau,x,z)g(\tau, z,y)dzd\tau.
 $$
Analogous notation is used  for the convolution of measures:
$$
(F* G)_t (du)= \int_\real F_t(du-z) G_t(dz), \quad (F\star G)_t(du)=\int_0^t \int_\real
F_{t-s}(du-z) G_s(dz)ds.
$$

\noindent We denote by $B_b(\real)$, $C_\infty(\real)$,
$C_b^k(\real)$, $C^\infty_0(\real)$ the spaces of functions, which are Borel measurable and bounded, continuous vanishing at infinity, $k$-times differentiable with bounded derivatives, infinitely smooth with compact support, respectively.
We write $L_x$ when we emphasize that the operator $L$ acts with respect to the variable $x$.

\section{Settings and the main result}\label{settings}

\subsection{Outline of the method}\label{outline}

To find the candidate for fundamental solution to \eqref{fund1} we use  parametrix method.  To do this  we involve the properties of the fundamental solution corresponding to the constant-coefficient operator, defined by \eqref{ld}.
It is known (see, for example, \cite{Ja01}) that the  operator $(L(D), C_\infty^2(\real))$ extends to  the generator of a convolution semigroup of probability measures, which gives rise to a L\'evy process $Z_t$. The characteristic function of this process is
\begin{equation}\label{char}
\Ee e^{i\xi Z_t}= e^{-tq(\xi)}, \quad t>0, \quad \xi\in \real,
\end{equation}
where  the function $q(\xi)$  admits the L\'evy-Khinchin representation
\begin{equation}\label{psi}
    q(\xi)=\int_\real (1-\cos(\xi u)) \mu(du).
    \end{equation}

In \cite{KK12a} it is shown that under  assumption \textbf{A1} (see below) this process admits the transition probability density, which can be written as
\begin{equation}\label{ptx}
p_t(x)= \frac{1}{2\pi} \int_\real e^{-ix\xi -tq(\xi)}d\xi.
\end{equation}
Put
\begin{equation}\label{pt0}
p_t^0(x,y):= p_t(y-x).
\end{equation}
Following the classical approach presented in  \cite{Fr64},  see also \cite{Ja02}, we are looking for the fundamental solution $p_t(x,y)$  to \eqref{fund1} in the form
 \begin{equation}\label{sol10}
    p_t(x,y)=p_t^0(x,y) +(p^0\star \Psi)_t(x,y),
    \end{equation}
where the function $\Psi: (0,\infty)\times \real\times \real\to \real$ is to be
determined.

Put
\begin{equation}\label{Phi}
\Phi_t(x,y):= \mathfrak{L}(x,D) p_t^0(x,y).
\end{equation}
Observe that since
\begin{equation}\label{p0}
\big(L(D)-\partial_t\big) p_t^0(x,y)=0,
\end{equation}
we get the equation for $\Psi_t(x,y)$:
$$
\Psi_t(x,y)=\Phi_t(x,y)+ (\Phi\star \Psi)_t(x,y),
$$
from where
\begin{equation}\label{Psi}
\Psi_t(x,y)=\sum_{k=1}^\infty \Phi^{\star k}_t (x,y),
\end{equation}
provided that the series converge. Thus, to justify representation \eqref{sol10} we need to show that the convolutions
$\Phi^{\star k}$ are well defined, and the series \eqref{Psi} converges.

The second part of the program consists of the justification procedure. We point out that the "classical" justification procedure from \cite{Fr64} cannot be performed in our situation.  Instead, in \cite{KK14a} we proposed a way how this problem can be handled in the case when the operator $L(x,D)$ is of the form
 $$
 a(x)  (-\Delta)^{-\alpha/2} + b(x)\nabla, \quad 0<\alpha<2.
 $$
In fact, this method  can be well adapted to our situation, which is done in Sections~\ref{cont} and \eqref{just}. See also \cite{KK14b}.

\subsection{Overview of the problem}\label{overview}

Let us briefly recall the background of the problem of
getting the transition density  estimates for Markov processes and
their applications.

The approach we are going to implement relies on the parametrix method, see
\cite{Fr64} for the description of the classical parametrix method for parabolic systems. Later on this method was extended in \cite{Dr77},  \cite{DE81}, \cite{Ko89}, and \cite{Ko00} to equations with pseudo-differential operators; see also the reference list and an  extensive overview in the monograph \cite{EIK04}.   See also  \cite{CZ13}, which  relied on the results from \cite{Ko89}, for the refined bounds for the constructed fundamental solution, and the martingale problem approach for the  justification procedure.
In  \cite{BJ07}  the case of the  fractional Laplacian perturbed by a gradient is treated;     in the justification procedure it is shown that the integro-differential operator is the weak generator of the respective semigroup.  In \cite{KS14} the case of singular perturbation of the fractional Laplacian is considered; see also \cite{CW13}  for another different approach, which relies on \cite{BJ07}.
We refer to  \cite{Po94} and \cite{Po95}
for the parametrix construction of the transition probability
density of the process which is the weak solution to the SDE driven by a
symmetric $\alpha$-stable process with a drift.  In  \cite{FP10}, \cite{DF13}  and \cite{KK14a} the gradient perturbation of an $\alpha$-stable like operator with $0<\alpha<1$ is investigated.

\begin{comment}
In  the  group of papers quoted above the strategy consists of two steps: 1) construction of the \emph{candidate} for being the fundamental solution to the respective integro-differential equation, 2) justification procedure, which shows that the constructed function is indeed the transition density of some Markov process, which is related in some sense to the initial operator. This steps  are realized  in the above papers in different ways; we refer to \cite{KK14a} for the extensive discussion on this topic.
\end{comment}

Another approach
to study the fundamental solution to \eqref{fund1}  involves a
version of  the parametrix method, which relies on the Hilbert space technique  and the symbolic calculus.  Such an  approach is
developed in \cite{Ts74}, \cite{Iw77}, \cite{Ku81}, \cite{Ho98a},
\cite{Ho98b}, \cite{Ja02}, and further extended to evolution
equations in \cite{B\"o05} and \cite{B\"o08}. In such a way one can
construct the  fundamental solution and show that it belongs to certain symbolic
class,  but sofar this
method does not give a way to construct explicit  estimates for
the solution.

 One can also investigate the question of existence and properties of the transition probability density of Markov processes using the Dirichlet forms approach. Starting with a symmetric regular Dirichlet
form and making the assumptions about the  absolute continuity of its kernel,
one can show that the transition probability density of the
respective Markov process exists and satisfies certain upper and
lower estimates, see \cite{CKS87}, \cite{CK08}, \cite{CKK08},
\cite{CKK10}, \cite{BBCK09}, \cite{BGK09}; of course, this list is
far from being complete.  In this case the justification procedure is in fact hidden in the construction itself: one has to
\emph{assume} that the Dirichlet form under consideration is regular.

\subsection{Main results}\label{main-res}

Let
    \begin{equation}
    q^U(\xi):=\int_{\real} \big((\xi u)^2\wedge 1\big)\mu(du),\quad q^L(\xi):=\int_{|u\xi|\leq 1} (\xi u)^2\mu(du).\label{psipm}
    \end{equation}
One can show that $q^L$ and $q^U$ satisfy the inequalities:
    \begin{equation}
    (1-\cos 1) q^L(\xi)\leq q(\xi) \leq 2q^U(\xi), \quad \xi\in \real. \label{psipm1}
    \end{equation}
In what follows we assume that $q(\xi)$ or, equivalently, the  symmetric L\'evy measure  $\mu$ related to $q(\xi)$ by \eqref{psi}, satisfies the condition \textbf{A1}  given below.

\medskip

\begin{itemize}

\item[\textbf{A1.}] There exists some $\beta>1$ such that for large $|\xi|$
\begin{equation}\label{A}q^U(\xi)\leq \beta q^L(\xi).
\end{equation}

\end{itemize}

 For example, for a symmetric $\alpha$-stable L\'evy measure $\mu(du):=c(\alpha) |u|^{-1-\alpha} du$, $\alpha\in (0,2)$, condition \textbf{A1} holds true with $\beta=2/\alpha$. For this reason we introduce the new index
  \begin{equation}\label{al}
  \alpha:= 2/\beta,
  \end{equation}
  where $\beta >1$ is the parameter from condition  \textbf{A1}.

  Condition  \textbf{A1} implies  that $q(\xi)$  has power growth for $|\xi|$ large enough, see  \cite{KK12a}:
\begin{equation}
q(\xi)\geq c|\xi|^{\alpha}. \label{growth}
\end{equation}
Note that the converse  is not true: the power growth type condition \eqref{growth} does not imply \textbf{A1}; see   \cite{KK12a} for the  detailed discussion.

Suppose that the function $m(x,u)$ from representation \eqref{Qxd} satisfies the assumptions below.

\medskip

\textbf{A2.} The function $m(x,u)\geq 0$ is symmetric with respect to $u$ for any $x\in \real$;

 \textbf{A3.}   $\sup_x m(x,u)\leq c(1\wedge |u|^\epsilon)$ for some $c$, $\epsilon>0$.

 \medskip

From now we assume that $\epsilon>0$ in condition \textbf{A3} is small enough, in particular,
\begin{equation}
\int (|u|^\epsilon \wedge 1) \mu(du)=\infty.\label{inf1}
\end{equation}
For example, \eqref{inf1} is satisfied when $0<\epsilon<\alpha$.

Below we state the first main result of our paper.

\begin{theorem}\label{t-main1}
Suppose that assumptions \textbf{A1} -- \textbf{A3} are satisfied. Then the  function $p_t(x,y)$  introduced in  \eqref{sol10} -- \eqref{Psi} is well defined.
\end{theorem}

\begin{remark}\label{rem10}
\begin{itemize}

\item[a)]
   If assumption \eqref{inf1} fails, the situation is even simpler. Heuristically, in this case the total intensity of the perturbation is finite. We investigate this situation in the Appendix~B.

\item[b)] By assumptions \textbf{A2} and \textbf{A3}, the measure $\mu(dy)$ dominates $m(x,u)\mu(du)$, implying that the operator $\mathfrak{L}(x,D)$ is a lower order perturbation of $L(D)$.

\item[c)]  The  symmetry assumption  imposed on the function  $m(x,u)$ and on the measure $\mu(du)$ is  purely technical, and  is introduced in order to make the presentation as transparent as possible. For a further investigation we refer to   \cite{KK14b}, where a more general kernel is considered.

%\item[d)]   One can relax   assumptions about the form of $\mu(x,du)$  and \textbf{A3} by assuming  the H\"older continuity in $x$ of the function $m(x,u)$, see \cite{KK13}.

  \end{itemize}
\end{remark}
By the theorem on continuity with respect to parameter, the functions $q^U(\xi)$ and $q^U(\xi)/\xi^2$ are continuous, respectively, on $[-1,1]$ and $\real\backslash [-1,1]$,  which implies that $q^U(\xi)$ is continuous on $\real$. Further, since $(q^U)'(\xi)= \frac{2}{\xi}q^L(\xi)$ in the a.e. sense, and  due to condition \textbf{A1} we have  $q^L(\xi)>0$ for all $\xi$ large enough, the function $q^U(\xi) $ is strictly increasing on $[a, \infty)$, where $a>0$ is some constant.
 Thus, the function
 \begin{equation}
\rho_t:=\big(q^U\big)^{-1}\left(1/t\right), \quad t\in (0,1], \label{rot}
\end{equation}
is correctly defined; here  $\big(q^U\big)^{-1}$  is the inverse  of $q^U$.

Observe that by \eqref{growth} we have
\begin{equation}
\rho_t \leq C  t^{-1/\alpha}, \quad t\in (0,1].\label{growth2}
\end{equation}
Denote by $\sigma\in [\alpha, 2]$  the minimal value for which there exists $c>0$ such that
\begin{equation}
\rho_t\geq c t^{-1/\sigma}, \quad t\in (0,1]. \label{growth3}
\end{equation}
This estimate is equivalent to the following upper bound on the growth of the characteristic exponent: there exists $c>0$ such that
$$
 q(\xi)\leq c |\xi|^\sigma \quad \text{for large $|\xi|$.  }
$$

Put
\begin{equation}
f_{up}(x):=d_1 e^{-d_2 |x|\log (1+|x|)},  \quad f_{low}(x) = d_3 (1-d_4|x|)_+, \label{flu}
\end{equation}
where $d_i$, $i=1-4$, are some positive constants.

In the proposition below we state the continuity and smoothness properties of the constructed function $p_t(x,y)$, and provide the respective upper bounds.
\begin{proposition}\label{p-main1}
\begin{itemize}
\item[1.] The function $p_t(x,y)$ is  continuous in $(t,x,y)\in (0,\infty)\times \real\times \real$.

\item[2.] There exist constants $d_1,d_2>0$ and a family of probability  measures $\{Q_t, \, t\geq 0\}$,  such that
\begin{equation}\label{eqII}
p_t(x,y)\leq \rho_t \Big(f_{up}(\rho_t \cdot ) * Q_t\Big)(y-x), \quad t\in (0,1], \quad x,y\in \real,
\end{equation}
where $f_{up}$ is a function of the form \eqref{flu} with constants $d_1,d_2$.

\item[3.] There exists  $\partial_t p_t(x,y)$, which is continuous in   $(t,x,y)\in (0,\infty)\times \real\times \real$.

 \item[4.] There exist constants $\tilde{d}_1,\tilde{d}_2>0$ and  a family of probability  measures $\{\tilde{Q}_t, \,t\geq 0\}$,  such that the following estimate holds true:
   \begin{equation}\label{der-es}
   |\prt_tp_t(x,y)|\leq Ct^{-1}\rho_t \Big(f_{up}* \tilde{Q}_t\Big)(y-x),\quad  t\in (0,1], \quad x,y\in \real,
   \end{equation}
where $f_{up}$ is a function of the form \eqref{flu} with constants $\tilde{d}_1,\tilde{d}_2$.

\end{itemize}
\end{proposition}

Proposition~\ref{p-main1} enables us to transfer the continuity and smoothness properties from $p_t(x,y)$ to the operator $T_t$, defined by \eqref{t-semi}.

\begin{proposition}\label{p-main2}

\begin{itemize}

\item[1.] The operator $T_t$, defined  in \eqref{t-semi}, maps $B_b(\real)$ into $C_\infty(\real)$.

\item[2.] For any $C_\infty(\real)$
\begin{equation}\label{Tt0}
T_t f \longrightarrow f, \quad t\to 0,
\end{equation}
in $C_\infty(\real)$.

\item[3.] For any $f\in C_\infty(\real)$  the mapping
   $$
   (0,\infty)\ni t\mapsto T_tf\in C_\infty(\real)
   $$
   is continuously differentiable, and its derivative is given by
   $$
   (\prt_t T_tf)(x)=\int_\real \prt_tp_t(x,y) f(y)dy.
   $$
\end{itemize}
\end{proposition}
Below we present the second main result of the paper.
%Along with Theorem~\ref{t-main1}, the theorem below is the second main result of our paper. It shows that  the  function $p_t(x,y)$  is indeed the transition probability density of a Markov process,  and clarifies the correspondence between this process and the operator $L(x,D)$.

\begin{theorem}\label{t-main2} Under  conditions of Theorem~\ref{t-main1}, the statements below hold true.

\begin{itemize}
  \item[I.]  The family of operators $(T_t)_{t\geq 0}$ defined in  \eqref{t-semi} forms  a strongly continuous conservative semigroup on  $C_\infty(\real)$, which corresponds to a (strong) Feller Markov process $X$.

   \item[II.]  The process $X$ is a solution to
   the martingale problem
  \begin{equation}\label{mart}
  (L, C^2_\infty(\real)).
  \end{equation}

  \item[III.] The closure in $C_\infty(\real)$ of the operator $(L(x,D), C_\infty^2(\real))$ is the generator of the semigroup $(T_t)_{t\geq 0}$.   Consequently,  the martingale problem \eqref{mart} is well posed, and the process $X$  is uniquely determined  as  its  unique solution.
\end{itemize}
\end{theorem}

Finally, we state the on-diagonal and lower bounds for $p_t(x,y)$.

\begin{proposition}\label{lower}
\begin{itemize}

\item[1.] There exist constants $c_1$, $c_2>0$ such that
\begin{equation}\label{eqIIa}
c_1 \rho_t \leq p_t(x,x)\leq c_2 \rho_t, \quad t\in (0,1], \quad x\in \real.
\end{equation}
\item[2.]  There exist constants $d_3,d_4>0$ such that
\begin{equation}\label{low}
p_t(x,y)\geq \rho_t  f_{low}((x-y)\rho_t), \quad t\in (0,1], \quad x,y\in \real,
\end{equation}
where $f_{low}$ is of the form  \eqref{flu} with constants $d_3,d_4$.
\end{itemize}

\end{proposition}

In the theorem below we show that under the assumption that the tails of the (re-scaled)  measure $\mu$ are dominated by the tails of some  distribution,  one can obtain the upper and lower estimates on $p_t(x,y)$ in a rather simple form.
\begin{definition} Let  $h:\, [0,\infty)\to[0,\infty)$.
We say that  $h\in \mathcal{L}$, if   $\lim_{x\to \infty} \frac{h(x-y)}{h(x)}=1$ for all $y>0$.
\end{definition}
\begin{theorem}\label{t-jump3}
  Assume that the conditions of Theorem~\ref{t-main1}  hold true, and $0<\epsilon<\alpha$ from condition
  \textbf{A3} is fixed. Suppose that one of the conditions below is satisfied:

\begin{itemize}

\item[I.] There exists a distribution function  $\mathfrak{G}(v)$ on $[0,\infty)$, such that
\begin{equation}
t\mu\Big(\{u: |\rho_t u|>v\}\Big)\leq C(1-\mathfrak{G}(v)), \quad v\geq 1, \quad t\in (0,1]. \label{dens02}
\end{equation}

\item[II.]  The L\'evy measure $\mu$ admits a density $\pi(u)$, and there exists a  probability density $\mathfrak{g}(u)$ on $[0,\infty)$ such that
\begin{equation}
\pi_t(u):=\tfrac{t}{\rho_t} \pi\big(\tfrac{u}{\rho_t})\leq c\mathfrak{g}(u), \quad u\geq 1, \quad t\in (0,1].  \label{dens01}
\end{equation}
\end{itemize}
In addition, assume that the function
\begin{equation}
h(x)=
\begin{cases}
1-\mathfrak{G}(x), & \text{under condition I;}
\\
\mathfrak{g}(x), & \text{under condition II, }
\end{cases} \label{h}
\end{equation}
belongs to $h\in \mathcal{L}$, $x^{2\epsilon} h(x)$ is monotone decreasing  for some $\epsilon>0$, and

 a)  for all $c,\, x\geq 1 $ we have  $h(cx)\leq c^{-1} h(x)$;

 b)  there exists $c>0$ such that $h(x)\leq c h(2x)$ for all $x\geq 1$.

Then the function $p_t(x,y)$, given by \eqref{sol10},   is well-defined,  and for all $t\in (0,1]$, $x,y\in \real$,  the
following estimate holds true:
\begin{equation}
p_t(x,y)\leq C\rho_t \big( f_{up}(\rho_t (y-x))+ h(\rho_t(y-x))+
t^{\epsilon/\sigma} h_\epsilon(\rho_t(y-x))\big) \label{txy2}
\end{equation}
where $h_\epsilon(x)=x^\epsilon
h(x)$, and $\sigma$ is defined prior to \eqref{growth3}.
\end{theorem}
% This theorem  is based on the results proved in \cite{KK12}.

\begin{remark}
Intuitively, Theorem~\ref{t-jump3} represents  the cases in which it
is possible to construct the "bell-like" estimate  (cf. \eqref{bell}) for the transition
probability density $p_t(x,y)$. In Section~\ref{example} we provide
some examples which illustrates the above theorems.  At the same
time, we emphasize, that although the bell-like estimate is more
explicit than the compound  kernel estimate proved in \eqref{eqII},
 the latter is more accurate, and reflects the
true  structure  of the impact of the L\'evy measure.
\end{remark}

\section{Construction of the parametrix series. Proofs of Theorems~\ref{t-main1} and \eqref{t-jump3}. }\label{constr}

\subsection{Generic Calculation}\label{gen}

In this subsection we state the results which are crucial for the proof of Theorems~\ref{t-main1} and Theorem~\ref{t-jump3}.

Let $g_t(x)$ be a function of the form
\begin{equation}\label{gtc}
g_{t}(x):=\rho_t e^{-c \rho_t|x| \ln (1+\rho_t|x|)},
\end{equation}
where $c>0$ is some constant.
Define also
%\begin{equation}\label{gtil}
%\tilde{g}_t(x)= t^{\delta}g_t(x),
%\end{equation}
\begin{equation}\label{hh}
h_\epsilon(x):= x^\epsilon h(x),\quad  h_{t,\epsilon}(x):= \rho_t h_\epsilon(|x|\rho_t),
\end{equation}
where $h$ is defined in \eqref{h}.

Suppose that we know already  that $\Phi_t(x,y)$ satisfies the upper bounds given below; the proofs will be given in
 Section~\ref{t21}.

\begin{itemize}

\item[i)] Under  the conditions of Theorem~\ref{t-main1},

\begin{equation} \label{ph11}
\big|   \Phi_t(x,y)\big|\leq C t^{-1+\delta} \big(\tilde{g}_{t} *
G_{t}\big)(y-x),
\end{equation}
where  $C>0$, $\delta\in (0,1)$ are some constant,
\begin{equation}\label{gttil}
\tilde{g}_{t}(x)= t^\delta g_t(x),
 \end{equation}
 the function  $g_t$ is of the form \eqref{gtc} with some constant $c>0$,  and  $\{G_t(du), t\geq 0\}$ is some  family of probability measures.

\item[ii)]
Under conditions of Theorem~\eqref{t-jump3},
\begin{equation}\label{ld101}
\big|\Phi_t(x,y)\big|\leq C t^{-1+\delta}\big( g_{t}(y-x) + h_{t,\epsilon}(y-x)\big),
\end{equation}
where  $C>0$, $\delta\in (0,1)$ are some constant,  and  $g_t$ is of the form  \eqref{gtc} with some constant $c>0$.
\end{itemize}

The key ingredient in the proof of the  upper bounds on the convolutions $\Phi_t^{\star k}(x,y)$ is provided by the convolution property of the functions $\tilde{g}_t$ and $h_{t,\epsilon}$, respectively.

Fix now the constant $c>0$ in the definition of $g_t$ in \eqref{gtc}, and put
\begin{equation}\label{gtk}
g_{t,\theta}(x):=  \rho_t e^{-\theta c \rho_t |x|\ln (1+\rho_t |x|)}, \quad \theta\in (0,1).
\end{equation}
\begin{lemma}\label{con1}
\begin{itemize}
\item[i)]
For any $\theta\in (0,1)$  one has
  \begin{equation}\label{gt11-new}
(\tilde{g}_{t-s} *\tilde{g}_{s} )(x) \leq (1-\theta)^{-1}c  t^{2\delta} g_{t,\theta}(x), \quad 0<s<t, \quad x\in \real.
\end{equation}
\item[ii)]  For any $\epsilon\in [0,\alpha)$ one has
\begin{equation}\label{h-es}
\big(h_{t-s,\epsilon}* h_{s,\epsilon}\big)(x)\leq C h_{t,\epsilon}(x), \quad x\in \real.
\end{equation}
\end{itemize}
\end{lemma}
We postpone the proof till Appendix C.

 For $k\geq 1$ define $G^{(1)}_t(dw)\equiv G_t(dw)$,
\begin{equation}\label{gk-new}
G^{(k+1)}_t(dw)={1\over B(\delta, k\delta)}\int_0^1 \int_\rn (1-r)^{-1+k\delta} r^{-1+\delta}  G_{t(1-r)}^{(k)}(dw-u) G_{tr}(du)dr.
\end{equation}

\begin{lemma}\label{gen}
Let  $(\theta_k)_{k\geq 1}$ be a sequence of real numbers, such that  $\theta_1=1$, $0<\theta_{k+1}<\theta_k$.
\begin{itemize}
\item[a)] Suppose that the estimate \eqref{ph11} holds true. Then
\begin{equation}
\big|\Phi^{\star k}_t(x,y)\big|\leq C_k t^{-1+k\delta}  \big(g_{t}^{(k)}* G_t^{(k)}\big)(x-y), \quad k\geq 2,  \label{lk0}
\end{equation}
where  $g_t^{(k)}(x)= t^{k\delta}g_{t,\theta_k}(x)$,
\begin{equation}\label{Ck}
C_k:=\frac{ C^k C_0^{k-1} \Gamma^k(\delta)}{\Gamma(k\delta)} \prod_{j=2}^k (\theta_{k-1}-\theta_k)^{-1},
\end{equation}
$C>0$ is the constant from \eqref{ph11},   $C_0>0$, and the family of  probability measures $\{G_t^{(k)},\, t>0,\, k\geq 2\}$ is defined by \eqref{gk-new}.

\item[b)]
Suppose that the estimate \eqref{ld101} holds true. Then
\begin{equation}
|\Phi^{\star k}_t(x,y)|\leq C_kt^{-1+k\delta}\big( g_{t,\theta_k}(x-y) + h_{t,\epsilon}(x-y)\big), \quad k\geq 2. \label{LZ21}
\end{equation}
where $C_k$ is given by \eqref{Ck}, in which now the constant $C>0$ comes from \eqref{ld101}.
\end{itemize}
\end{lemma}

\begin{proof}  a) We use induction.  Under \eqref{lk0}  and  \eqref{ph1} we have
    \begin{equation}\label{con20}
    \begin{split}
     \big|\Phi^{\star k}_t(x,y)\big|&\leq C C_{k-1} \int_0^t \int_\real (t-s)^{-1+(k-1)\delta} s^{-1+\delta}
\\
&\quad \cdot  \big(g_{t-s}^{(k-1)} * G_{t-s}^{ (k-1)}\big)(x-z)\big(g_s^{(1)}* G_s\big)(z-y)dz \, ds\\
 &\leq  C C_{k-1} t^{k \delta} \int_0^t \int_\real (t-s)^{-1+(k-1)\delta} s^{-1+\delta}\\
 &\cdot \big(g_{t-s,\theta_{k-1}}* G_{t-s}^{(k-1)} \big)(x-z)\big(g_{s,\theta_{k-1}}* G_s\big)(z-y)dz \, ds.
\end{split}
    \end{equation}
where in the last line we used the monotonicity of $g_t(x)$ in $x$.

 Using Lemma~\ref{con1}  we get
$$
\big( g_{t,\theta_{k-1}}* g_{t,\theta_{k-1}}\big)(z)\leq C_0 \big(\theta_{k-1}-\theta_k\big)^{-1} g_{t,\theta_k}(z),
$$
where $C_0>0$ is some constant, and $\theta_k\in (0,\theta_{k-1})$.

 Therefore, making the change of variables, we derive
\begin{equation}
\begin{split}
\big|\Phi_t^{\star k}(x,y)|&%\leq B_k B_1 \int_0^t \int_\real (g_{t-s,\theta_k} *
%G^{\star k}_{t-s}(x-z) g_{s,\theta_1} * G_{t_k}(z-y)dzds
%\\&
\leq   C C_{k-1} C_0 (\theta_{k-1}-\theta_k)^{-1} t^{1+2k\delta}  \int_\real  g_{t,\theta_{k}}(x-y-w_1-w_2)\\
&\cdot \Big[
\int_0^1 \int_\real (1-r)^{-1+(k-1)\delta} r^{-1+\delta} G^{ (k-1)}_{t(1-r)} (dw_1) G_{tr}(dw_2)dr\Big]
\\&
=C C_0 (\theta_{k-1}-\theta_k)^{-1} B((k-1)\delta, \delta) C_{k-1} t^{1+k\delta}  \big( g_t^{(k)} *  G_t^{( k)}\big)(y-x),
\end{split}
\end{equation}
where
$$
C_k:= C_0C (\theta_{k-1}-\theta_k)^{-1} B((k-1)\delta, \delta) C_{k-1}.
$$
By induction, we obtain the expression for  $C_k$ as in \eqref{Ck}.

\medskip

b)   We have:
\begin{align*}
\big|\Phi_t^{\star k}(x,y)| & \leq  C_{k-1}C \Big\{\int_0^t\int_\real  (t-s)^{-1+(k-1)\delta}s^{-1+\delta} g_{t-s,\theta_{k-1}}(x-z)  g_{s}(z-y)  \,dzds
\\
\quad &+ \int_0^t\int_\real  (t-s)^{-1+(k-1)\delta}s^{-1+\delta} g_{t-s,\theta_{k-1}}(x-z)   h_{s,\epsilon}(z-y)\,dzds\\
\quad &+ \int_0^t\int_\real  (t-s)^{-1+(k-1)\delta}s^{-1+\delta} h_{t-s,\epsilon}(x-z) g_{s,\theta_{k-1}}(z-y)   \,dzds
\\
\quad &+  \int_0^t\int_\real  (t-s)^{-1+(k-1)\delta}s^{-1+\delta} h_{t-s,\epsilon}(x-z) h_{s,\epsilon}(z-y)\,dzds\Big\}
\\&
= CC_{k-1} (I_1+I_2+I_3+I_4).
\end{align*}
 The term  $I_1$ can be estimated in the same way as in part a) of the lemma.  Namely, by Lemma~\ref{con1} we get:
\begin{equation}
\begin{split}
I_1&\leq c (\theta_{k-1}-\theta_k)^{-1}  g_{t,\theta_{k}} (x-y) \int_0^t (t-s)^{-1+(k-1)\delta}s^{-1+\delta}ds\\
& \leq c(\theta_{k-1}-\theta_k)^{-1}  B((k-1)\delta,\delta) t^{-1+k\delta}g_{t,\theta_{k}} (x-y),
\end{split}
\end{equation}
where $\theta_{k}\in (0,\theta_{k-1})$.   Using Lemma~\ref{con1}.ii) we get
\begin{align*}
I_4&\leq c h_{t,\epsilon} (x-y) \int_0^{t} (t-s)^{-1+(k-1)\delta} s^{-\delta}ds
\leq    c B((k-1)\delta,\delta)  t^{-1+k\delta}  h_{t,\epsilon} (x-y).
\end{align*}
Next we estimate $I_2$, the estimate for $I_3$ can be obtained in the same way.

Let us estimate the inner integral in  $I_2$. Suppose first that $0<s<t/2$.  We use similar argument as in the proof of Lemma~\ref{con1}.ii).
Suppose  that $|x-y-z|> |x-y|/2$.  Then
\begin{align*}
\int_{|x-y-z|> |x-y|/2}  g_{t-s,\theta_{k-1}}(x-y-z)h_{s,\epsilon}(z-y)dy& \leq c_1 g_{t,\theta_{k-1}}((x-y)/2)\int_\real h_{s,\epsilon}(z)dz \\
&\leq c_2 g_{t,\theta_k}(x-y),
\end{align*}
where $\theta_k<\theta_{k-1}/2$.
Further, since the inequality $|x-y-z|\leq |x-y|/2$ implies $|z|\geq |x-y|/2$, we get
\begin{align*}
\int_{|x-y-z|\leq  |w|/2}  g_{t-s,\theta_{k-1}}(x-y-z)h_{s,\epsilon}(z)dy& \leq h_{s,\epsilon}(|x-y|/2)   \int_\real g_{t-s,\theta_{k-1}}(z)dz \\
&\leq c_1 \theta_{k-1}^{-1} h_{t,\epsilon}(x-y),
\end{align*}
where for the last inequality we used assumptions a) and b)  on of Theorem~\ref{t-jump3}.  Suppose now that $t/2\leq s\leq t$. By monotonicity of $\rho_t$, the inequality $\rho_{t-s}|z|\leq 1$ implies $\rho_s|z|\leq 1$, and $\rho_s |x-z|\leq \rho_s |x|+1$. Then by monotonicity of $h$ we get
$$
\int_{\rho_{t-s}|z|\leq 1}  g_{t-s,\theta_{k-1}}(z)h_{s,\epsilon}(x-z)dz\leq c_1 h_s (x)\int_\real g_{t-s,\theta_{k-1}}(z) dz\leq c_2 \theta_{k-1}^{-1}h_t(x).
$$
Note that in the domain $\rho_{t-s}|z|\geq 1$ we have
$g_{t-s,\theta_{k-1}} (z)\leq c_3 h_{s,\epsilon}$. Then
\begin{align*}
\int_{\rho_{t-s}|z|> 1}  g_{t-s,\theta_{k-1}}(z)h_{s,\epsilon}(x-z)dz&\leq \int_\real h_{t-s,\epsilon}(z)h_{s,\epsilon}(x-z)dz\leq c_4 h_{t,\epsilon}(x).
\end{align*}

Substituting the above estimates in $I_2$, we get
$$
I_2\leq  c B((k-1)\delta,\delta)  t^{-1+k\delta} \Big[ \theta_{k-1}^{-1}  h_{t,\epsilon} (x-y)+ g_{t,\theta_k}(x-y)\Big].
$$
Thus, we obtain
\begin{align*}
\big|\Phi_t^{\star k}(x,y)| &\leq C_k   t^{-1+k\delta} \Big[g_{t,\theta_k}(x-y)+ h_{t,\epsilon} (x-y)\Big]
\end{align*}
with $C_k= 4c C C_{k-1} (\theta_{k-1}-\theta_k)^{-1}B((k-1)\delta, \delta)$. By induction, we can write $C_k$ as  in \eqref{Ck}.
\end{proof}

Note that estimates \eqref{lk0} and \eqref{LZ21} are still not sufficient for proving the convergence of the series  $\sum_{k=1}^\infty \Phi_t(x,y)$, because constants $C_k$ depend on $k$ in a rather complicated way;  for example, if we chose $\theta_k=\frac{1}{2}+\frac{1}{2k}$, it can be shown that $C_k\to\infty$ as $k\to \infty$. In order to overcome this problem, let us look more closely on the the right-hand side of \eqref{lk0} and \eqref{LZ21}, respectively.

Take
\begin{equation}\label{k0}
k_0:= \Big[\frac{n}{\alpha \delta}\Big]+1.
 \end{equation}
 For such $k_0$ we  have $t^{k_0 \delta}\rho_t^n \leq c $ for all $t\in [0,1]$. Then
 $$
 \Big(g_{t-s}^{(k_0)} * g_s^{(1)}\Big)(x)\leq  c(k_0) M_1 \rho_t^{-1} g_{t,\zeta}(z),
 $$
 where $\zeta=c\theta_{k_0}$,  and
\begin{equation}\label{M1}
 M:= \int_\real e^{-c (1-\theta_{k_0}) |z|}dz.
 \end{equation}
 Therefore, we obtain the following lemma.

 \begin{lemma}\label{es-new} Let $k_0$ be given by \eqref{k0}.
 \begin{itemize}
 \item[I.] Under conditions of Theorem~\ref{t-main1},
 \begin{equation}\label{LZkl}
\big| \Phi_t^{\star (k_0+\ell)}(x,y) \big|\leq  D_\ell  t^{-1+ \delta (k_0+\ell)}\rho_t^{-1} \big(g_{t,\zeta} * G_t^{(k_0+\ell)} \big) (y-x), \quad \ell\geq 1,
\end{equation}
where $C(k_0)$ is some constant, the family of probability measures $\{ G_t^{(k)},\, t>0, \, k\geq 1\}$ is defined in \eqref{gk-new},
\begin{equation}\label{Dl}
D_\ell :=  C(k_0) (C M)^\ell B((k_0+\ell -1)\delta, \delta),
\end{equation}
where  $M>0$ is given by \eqref{M1}, and $C>0$ is the constant, appearing in  \eqref{ph11}.

\item[II.] Under conditions of Theorem~\ref{t-jump3},
\begin{equation}\label{LZkl}
\big| \Phi_t^{\star (k_0+\ell)}(x,y) \big|\leq  D_\ell  t^{-1+ \delta (k_0+\ell)}  \rho_t^{-1} \big[ g_{t,\zeta} (y-x)+ h_{t,\epsilon}(y-x)\big],   \quad \ell\geq 1,
\end{equation}
where $D_\ell$ is given by \eqref{Dl} with
\begin{equation}\label{M2}
M:= \max\Big( \int_\real e^{-c (1-\theta_{k_0}) |z|}dz, \int_\real h_{\epsilon}(z)dz\Big),
 \end{equation}
and  $C>0$ is the constant, appearing in  \eqref{ld101}.
\end{itemize}

\end{lemma}
\begin{proof}
The proof is obtained by induction in the same manner as the proof of Lemma~\ref{gen}; we only need to use the inequalities
$$
 \Big(\rho_{t-s}^{-1}g_{t-s}^{(k_0)} * g_s^{(1)}\Big)(x)\leq  c(k_0) M \rho_t^{-1} g_{t,\zeta}(z),
 $$
$$
 \Big(\rho_{t-s}^{-1}h_{t-s,\epsilon} * h_{s,\epsilon}\Big)(x)\leq  c(k_0) M \rho_t^{-1} h_{t,\epsilon}(z),
 $$
where the constant $M$ is given by \eqref{M1} or \eqref{M2}, respectively. We omit the details.
\end{proof}

\subsection{Estimation of $\Phi_t(x,y)$}\label{t21}

In this section we derive the upper bound on $\Phi_t(x,y)$ under conditions of Theorem~\ref{t-main1} and
\ref{t-jump3}, respectively.

Put
\begin{equation}
\Lambda_t(du):=t  \mu(du)1_{|\rho_tu|>1}, \label{lam}
\end{equation}
and   define the measure
\begin{equation}
    P_t(dw) :=e^{-\Lambda_t(\real)} \sum_{k=0}^\infty \frac{1}{k!} \Lambda_t^{*k}(dw). \label{Poist}
    \end{equation}
Note that $\Lambda_t(\real)\leq t q^U(\rho_t)=1$.

For some  $0<\epsilon<\alpha$ define
\begin{equation}
\chi_{t,\epsilon}(du): = \rho_t^\epsilon \big(|u|^\epsilon\wedge
1\big) \Lambda_t(du), \label{chi1}
\end{equation}
\begin{equation}
G_{t}(du):= c_0\big( P_t(du)+ (P_t*
\chi_{t,\epsilon})(du)\big).\label{G1}
\end{equation}
Here $c_0>0$ is the normalizing constant, chosen in such a way that $G_t(\real)=1$.
\begin{lemma}\label{Phi-up}
Under conditions of Theorem~\ref{t-main1}  we have
\begin{equation} \label{ph1}
\big|   \Phi_t(x,y)\big|\leq C t^{-1+\eta} \big(g_{t} *
G_{t}\big)(y-x),
\end{equation}
where  $C>0$ is some constant, $g_{t}$ is of the form \eqref{gtc} with some constant $c>0$,  and  $\{G_t(du), t\geq 0\}$ is the  family of probability measures, given by \eqref{G1}, and $\eta=\epsilon/\sigma$.
\end{lemma}

\begin{lemma}\label{Phi-up2}
Under conditions of Theorem~\eqref{t-jump3}, we have
\begin{equation}\label{ld10}
\big|\Phi_t(x,y)\big|\leq C t^{-1+\eta}\big( g_{t}(x-y) + h_{t,\epsilon}(x-y)\big),
\end{equation}
where $C>0$ is some constant, $\eta=\epsilon/\sigma$, and  $g_t$ is of the form  \eqref{gtc} with some  constant $c>0$.
\end{lemma}

The proof relies on a few auxiliary statement from \cite{KK12a}, which we give below.
%Below we introduce the functions and kernels, which appear in the compound kernel estimates for $p(t,x,y)$.

Define
\begin{equation}\label{ft}
f_t(x):=\int_\real \rho_t f_{up}((x-w)\rho_t ) P_t(dw), \quad x\in \real, \quad t\in (0,1].
\end{equation}

\begin{lemma}[\cite{KK12a}]\label{aux1}
Suppose that the measure $\mu$ satisfies condition \textbf{A1}.   Then the assertions below hold true.
\begin{itemize}
\item[a)]
The L\'evy process $Z_t$ related to $\mu$ by \eqref{char}--\eqref{psi}  admits the transition probability density \eqref{ptx}, which belongs with respect to $x$ to $C_\infty^k(\real)$, $k\geq 0$, and   the derivatives satisfy
\begin{equation}\label{ptx-der}
\Big| \frac{\partial^k}{\partial x^k} p_t (x)\Big|
 \leq   \rho_t^k f_t(x), \quad t\in (0,1],\quad  x\in \real.
\end{equation}
Here $f_t$  is the function of the form \eqref{ft}, $f_{up}$ is of the form \eqref{flu}
with constants $A_k$ and $a_k$ in place of $d_1$ and $d_2$, respectively.

\item[b)] The lower bound holds true:
$$
p_t(x)\geq  \rho_t f_{low}(x\rho_t), \quad  t\in (0,1], \quad x\in \real.
$$
Here $f_{low}$ is the function of the form \eqref{flu}  with some constants $d_3,d_4>0$.
\end{itemize}

\end{lemma}

One can construct  more explicit (but not necessarily more precise)  estimates on the derivatives at the price of more restrictive assumptions on the L\'evy measure.
Recall the definition of a sub-exponential probability measure  and a sub-exponential probability density.
\begin{definition}\label{sub} A distribution function  $\mathfrak{G}$ on $[0,\infty)$ is called  \emph{sub-exponential}, if

(i) for  every $y\in \real$ one has $\underset{x\to\infty}{\lim}\frac{1- \mathfrak{G}(x-y)}{1- \mathfrak{G}(x)} =1$;

(ii) $ \underset{x\to\infty}{\lim} \frac{1- \mathfrak{G}^{*2}(x)}{1- \mathfrak{G}(x)}=2$.

A distribution density $\mathfrak{g}$  on $[0,\infty)$ is called  \emph{sub-exponential}, if it is positive on $[x_0, \infty)$ for some $x_0\geq 0$, and

(i) for  every $y\in \real$ one has $\underset{x\to\infty}{\lim}\frac{ \mathfrak{g}(x-y)}{\mathfrak{g}(x)} =1$;

(ii) $\underset{x\to\infty}{\lim} \frac{\mathfrak{g}^{*2}(x)}{\mathfrak{g}(x)}=2$.

\end{definition}
We refer to  \cite{EGV79} and  \cite{Kl89} for the basic properties of sub-exponential distribution functions  and distribution densities.

We quote a result from  \cite{KK12a}  on the upper estimates on the transition  probability density $p_t(x)$ of $Z$ under the assumption of sub-exponentiality of the tails of the L\'evy measure. In our notation, this statement reads as follows.
\begin{lemma} \label{aux2}    Suppose that  one of the conditions below hold true:

i) there exists a sub-exponential distribution function $\mathfrak{G}$ on $[1,\infty)$ such that the measure $\mu(du)$ satisfies $t\mu\big(\{u: |\rho_t u|>v\}\big)\leq C(1-\mathfrak{G}(v))$ for $ v\geq 1 $, $t\in (0,1]$;

ii) the measure $\mu(du)$ possesses a density  $\pi(u)$  with respect to the Lebesgue measure, and there exists a sub-exponential density $\mathfrak{g}$ on $[1,\infty)$ such that  $\tfrac{t}{\rho_t}\pi\big(\tfrac{u}{\rho_t})\leq c\mathfrak{g}(u)$, $u\geq 1$, $ t\in (0,1]$.
Then
\begin{equation}\label{ptx20}
\Big| \frac{\partial^k}{\partial x^k} p_t(x)\Big| \leq c_k \rho_t^k \big( f_{up}(x\rho_t) + h(x\rho_t)\big), \quad t\in (0,1], \quad x\in \real,
\end{equation}
where the function $h$ is defined in \eqref{h}, and $f_{up}$ is of the form \eqref{flu} with some constant $a_k$ in the place of $d_2$ and $d_1=1$.
\end{lemma}

%\subsection{Proof of Theorem~\ref{t-main1} }\label{t21}

The proof relies on Lemma~\ref{Phi-up} and on  the first statement of Lemma~\ref{gen}.  Let us introduce the objects which will be used in the proofs below.

\begin{comment}
\begin{equation}
    G_{t,\epsilon}^{(k)}(dw):=
    \begin{cases}
    G_{t,\epsilon}(dw),& k=1,
    \\
    G^{\star k}_{t,\epsilon}(dw), &  2\leq  k\leq k_0,
    \\
     (\rho_t  G^{(k_0)}_{t,\epsilon})  \star G^{\star (k-k_0)}_{t,\epsilon}(dw), & k
    > k_0+1,
    \end{cases}\label{Gk}
    \end{equation}
where $\sigma$ is defined prior to  \eqref{growth3},  and
\begin{equation}\label{k0}
k_0:=\Big[ \frac{\sigma(1+\alpha)}{\alpha\epsilon} \Big]+1.
\end{equation}
Put
\begin{equation}
\Pi_{t,\epsilon}(dw) := \sum_{k=1}^\infty A^k  G_{t,\epsilon}^{(k)} (dw), \label{pi}
\end{equation}
where $A>1$ is some constant,  and
\begin{equation}
Q_t(dw):=P_t(dw) + \big(P_t \star
\Pi_{t,\epsilon}\big) (dw).   \label{Q1}
\end{equation}

\end{comment}

\begin{proof}[Proof of Lemma~\ref{Phi-up}]
Since $p_t^0(x,y)$   satisfies \eqref{p0},  we have
    \begin{align*}
    \Phi_t(x,y)&=
    %[\partial/\partial t -L(x,D)\pm % L(y,D)]p_{t,y}(x)
    %\\
    \int_\real [p_t^0(x+u,y)-p_t^0(x,y)]m(y,u) \mu(du)
    \\&
    =
    \left[ \int_{|\rho_t  u|\leq 1}+\int_{|\rho_t u|>1} \right] [p_t^0(x+u,y)-p_t^0(x,y)] m(y,u) \mu(du)
    \\&
    =:J_1+J_2.
    \end{align*}
First we estimate $J_1$. Due to symmetry of the measure $\mu$ and symmetry of $m(x,u)$ in $u$, we can write $J_1$ as
$$
J_1= \int_{|\rho_t u|\leq 1}  [p_t^0(x+u,y)-p_t^0(x,y) -u\frac{\partial}{\partial x} p_t^0(x,y)] m(y,u) \mu(du).
$$
 Using the Taylor expansion with the remaining term in the integral form, we get
\begin{align*}
 \Big|p_t^0(x+u,y)-p_t^0(x,y)- u\frac{\partial}{\partial x} p_t^0(x,y) \Big| & = \Big| \int_0^u \tfrac{\partial^2}{\partial v^2} p_t^0(x+v,y) (u-v) dv\Big|
 \\&
 \leq  u^2 \Theta(u,t,y-x),
\end{align*}
where
$$
\Theta(u,t,y-x):= \frac{1}{|u| }\int_0^u\Big| \tfrac{\partial^2}{\partial v^2} p_t^0(x+v,y)\Big|dv.
$$
Let us estimate $\Theta(u,t,x)$. For simplicity, we assume that $u>0$, the case  $u<0$ is analogous. Using Lemma~\ref{aux1} to estimate  $\tfrac{\partial^2}{\partial v^2} p_y^0(x+v,y)$   and performing the change of variables, we get
\begin{equation}
\begin{split}
\Theta(u,t,x) &\leq  \frac{1}{u } \int_0^{u} \int_\real \rho_t^3 f_{up} ((x-w+v)\rho_t)P_t(dw) dv
\\&
=  \frac{\rho_t^2}{u } \int_\real \int_0^{u \rho_t}  f_{up} ((x-w)\rho_t-z)dz P_t(dw).
\end{split}\label{TH1}
\end{equation}
Note that in $J_1$ we integrate in $u$ over the domain $\{ u:\,\, |u\rho_t|\leq 1\}$. Let us show that there exist $c>0$ and $\vartheta\in (0,1)$, independent of $y$,  such that for all  $V\in [0,1]$
\begin{equation}
\frac{1}{V } \int_0^{V } f_{up} (y-z)dz\leq c f_{up}^{\vartheta} (y).\label{V1}
\end{equation}
Let $y\leq 0$. Then for $z\geq 0$ we have $f_{up}(y-z)= d_1e^{-d_2 (|y|+z)\ln (|y|+z+1)}\leq f_{up}(z)$, and \eqref{V1} follows.
Let $y\in [0,2]$. Since there exists $c\in (0,1)$ such that $c\leq e^{-d_2 |y|\ln (|y|+1)}\leq 1$ for all $y\in [0,2]$, then
$$
\frac{1}{V } \int_0^{V } f_{up} (y-z)dz\leq d_1\leq c^{-1} f_{up}(y), \quad y\in [0,2].
$$
Finally, let $y\geq 2$. Take $\vartheta<\min_{y\geq 2} \frac{(y-1)\ln y}{y\ln (y+1)}$. Then for any $z\in [0,1]$ we get
$$
(y-z)\ln (y-z+1)\geq (y-1)\ln y \geq \vartheta  y \ln (y+1),
$$
which implies  \eqref{V1}.
Thus, for all $y\in \real$ we have \eqref{V1} with $c$ and $\vartheta $ as above.

Using \eqref{V1} for estimation of the right-hand side of \eqref{TH1}, we get for  all $u$ such that  $|u\rho_t|\leq 1$  the estimate
$$
|\Theta(u,t,x)|\leq c_3 \rho_t^2  \big(g_{t,\vartheta }* P_t\big)(x),
$$
where $g_{t,\vartheta}(x)$ is defined in \eqref{gtk}, and $\vartheta\in (0,1)$ comes from \eqref{V1}.

To complete the estimation of $J_1$ it remains to  estimate  the integral $\rho_t^2 \int_{|\rho_t u |\leq 1} |u|^{2+\epsilon} \mu(du)$. We have:
\begin{align*}
\rho_t^2 \int_{|\rho_t u |\leq 1} |u|^{2+\epsilon} \mu(du)&= \frac{1}{\rho_t^\epsilon } \int_{|\rho_t u |\leq 1} |\rho_t u|^{2+\epsilon} \mu(du)
\leq \frac{1}{\rho_t^\epsilon } \int_{|\rho_t u |\leq 1} |\rho_t u|^2 \mu(du)
\\&
=\frac{1}{\rho_t^\epsilon}  q^L(\rho_t )
\leq \frac{c_4}{\rho_t^\epsilon} q^U(\rho_t )
\\&
\leq \frac{c_4}{t\rho_t^\epsilon},
\end{align*}
where in the last line we used that $q^U(\rho_t)=1/t$.
%Note that we always have
%\begin{equation}
%\rho_t\geq c_6 t^{-1/2}. \label{12}
%\end{equation}
Let $\eta:= \epsilon/\sigma$, where $\sigma$ is defined prior to \eqref{growth3}; then $\tfrac{1}{t \rho_t^\epsilon}\leq c_5 t^{-1+\eta}$.

Thus, from the above calculations, we obtain
\begin{equation}
\begin{split}
J_1 &\leq \int_{|\rho_t u|\leq 1} |u|^{2+\epsilon}   \Theta(u,t,x) \mu(du)
\leq  \rho_t^2  \int_{|\rho_t u|\leq 1} |u|^{2+\epsilon} \mu(du)\,\big(g_{t,\vartheta}* P_t\big)(x)
\\&
\leq  c_5 t^{-1+\eta}  \big(g_{t,\vartheta}* P_t\big)(x).
\end{split}\label{j01}
\end{equation}
Let us   estimate $J_2$.
Recall the measure $\chi_{t,\epsilon}(du)$ defined in \eqref{chi1}:
$$
\chi_{t,\epsilon}(du)= t\rho_t^\epsilon  (|u|^\epsilon\wedge 1) \I_{\{|\rho_t u|>1\}}\mu(du).
$$
Let us show that $\chi_{t,\epsilon}(\real)\leq c <\infty$ for all $t\in [0,1]$.  By our assumption that $0<\epsilon<\alpha$ we have
\begin{align*}
\chi_{t,\epsilon} (\real) &= 2t \rho_t^\epsilon \int_{1/\rho_t}^1 u^\epsilon \mu(du)
+ 2 t\rho_t^\epsilon \mu[1,\infty)
\\&
\leq 2t\epsilon \rho_t^\epsilon \int_1^{\rho_t} \frac{\mu\{ u: \, us \geq 1\}}{s^{1+\epsilon}} ds+ 2 t^{1-\alpha\epsilon}\mu[1,\infty)
\\&
\leq 2t\epsilon \rho_t^\epsilon \int_1^{\rho_t} \frac{q^U(s)}{s^{1+\epsilon}}ds+2\mu[1,\infty).
\end{align*}
Note that since $(q^U(s))' = 2s^{-1} q^L(s)$ in the a.e. sense,  we have by the l'Hospital rule
$$
\lim_{r\to\infty} \frac{\int_1^r s^{-1-\epsilon}q^U(s)ds}{r^{-\epsilon}q^U(r)}
=
\lim_{s\to\infty}\frac{q^U(r)}{2q^L(r)-\epsilon q^U(r)}\leq \frac{1}{\alpha-\epsilon},
$$
implying that
$$
t \rho_t^\epsilon \int_1^{\rho_t} \frac{q^U(s)}{s^{1+\epsilon}}ds\leq c_1 t q^U(\rho_t) =c_1,
$$
which proves our claim that  $\chi_{t,\epsilon}(\real)\leq c<\infty$ for all $t\in (0,1]$.
 Thus,  we have
    \begin{equation}
    \begin{split}
    |J_2|&\leq \int_{|\rho_t  u|> 1}|p_t^0(x+u,y)-p_t^0(x,y)|m(y,u)\mu(du)
    \\&
    \leq   c_1 \int_{|\rho_t u|>1} p_t^0(x+u,y) (|u|^\epsilon\wedge 1) \mu(du) + c_1 p_t^0(x,y) \int_{|\rho_t u|>1}(|u|^\epsilon\wedge 1) \mu(du)
    \\&
    = c_2 t^{-1+\eta} \big(p_t^0 * \chi_{t,\epsilon}\big)(y-x) + c_2 t^{-1+\eta} \chi_{t,\epsilon}(\real) p_t^0(x,y)
    \\&
   % \leq  c_7 t^{-\delta}  \big(g_{t,0} * G_t\big) (x),
    \leq c_3 t^{-1+\eta} \left( g_{t}* (P_t* \chi_{t,\epsilon} + P_t)\right)(x).
        \end{split}
    \end{equation}
Thus,  we arrive at \eqref{ph1} with $\eta=\epsilon/\sigma$,  some constant $C>0$, $g_t$ of the form \eqref{gtc}, and $G_t(dw)$  given by \eqref{G1}.

\end{proof}

\begin{proof}[Proof of Lemma~\ref{Phi-up2}]  The proof relies on Lemma~\ref{aux2}.
Observe that under  the conditions on the function $h$ posed in the theorem, the distribution function $\mathfrak{G}$ (resp., the distribution density $\mathfrak{g}$)  is sub-exponential. Indeed,  condition i) from Definition~\ref{sub} is clearly satisfied; for ii) we have by i) and the dominated convergence theorem
\begin{align*}
\lim_{x\to\infty} \frac{1-\mathfrak{G}^{*2}(x)}{1-\mathfrak{G}(x)}&=\lim_{x\to\infty}\int_1^{x-1} \frac{1-\mathfrak{G}(x-y)}{1-\mathfrak{G}(x)} d\mathfrak{G}(y)+ \lim_{x\to\infty} \frac{1-\mathfrak{G}(x-1)}{1-\mathfrak{G}(x)}
\\&=2,
\end{align*}
when the condition  I of Theorem~\ref{t-jump3} holds true,  and
\begin{align*}
\lim_{x\to\infty} \frac{\mathfrak{g}^{*2}(x)}{\mathfrak{g}(x)}&=\lim_{x\to\infty}\int_1^{x/2} \frac{\mathfrak{g}(x-y)\mathfrak{g}(y)}{\mathfrak{g}(x)} dy+ \lim_{x\to\infty}\int_{x/2}^{x-1} \frac{\mathfrak{g}(x-y)\mathfrak{g}(y)}{\mathfrak{g}(x)} dy
\\&
= 2\lim_{x\to\infty}\int_1^{x/2} \frac{\mathfrak{g}(x-y)\mathfrak{g}(y)}{\mathfrak{g}(x)} dy
\\&=2,
\end{align*}
in the case when the condition  II of Theorem~\ref{t-jump3} holds true; in the second line  of the last display  we used the change of variables.

Note that since $h_{2\epsilon}(x)\equiv x^{2\epsilon}h(x)$ is monotone decreasing,  $\mathfrak{G}_\epsilon(v) := 1- v^\epsilon(1-\mathfrak{G}(v))$   and  $\mathfrak{g}_\epsilon(v):= c_\epsilon v^\epsilon \mathfrak{g}(v)$ are, respectively, the distribution function and the distribution density; here $c_\epsilon>0$ is the normalizing constant.

Observe, that the function $\chi_{t,\epsilon}$ defined in \eqref{chi1} satisfies
\begin{equation}
\chi_{t,\epsilon}\{u:\,\, |u\rho_t|\geq v\}\leq c h_\epsilon(v), \quad v\geq 1. \label{chi10}
\end{equation}
 Indeed,  suppose first that $v\geq \rho_t$. Then  by \eqref{dens01} and \eqref{dens02} we have
\begin{align*}
\chi_{t,\epsilon} \{ u:\,\, |u\rho_t|\geq v\} &= t \rho_t^\epsilon \mu\{ u:\,\, |u \rho_t|\geq v\} \I_{\{\rho_t\leq v\}} \leq c_1 \rho_t^\epsilon h(v)\I_{\{\rho_t\leq v\}}
\leq c_1 v^\epsilon h(v).
\end{align*}
Similarly, for $v\leq \rho_t$ we have
\begin{align*}
\chi_{t,\epsilon} \{ u:\,\, |u\rho_t|\geq v\} &\leq  t \int_{v^\epsilon}^\infty \mu\{ u:\,\, |u \rho_t|^\epsilon \geq r\} dr
 = \epsilon t \int_v^{\infty}  \frac{\mu\{ u:\,\, |u \rho_t| \geq r\} }{r^{1-\epsilon}} dr
 \\&
 \leq  c_1 \epsilon\int_v^\infty \frac{r^{2\epsilon}  h(r)}{r^{1+\epsilon}}dr\leq c_2 \epsilon v^{2\epsilon} h(v)\int_v^\infty \frac{dr}{r^{1+\epsilon}}
 \\&
 \leq c_2 v^\epsilon h(v),
\end{align*}
where in the second line from below we used $h_{2\epsilon}(x)$ is decreasing.

Therefore, subexponentiality of $h$ (and, hence, of $h_\epsilon$)  we derive in both cases I and II
\begin{align*}
g_{t,\vartheta}* (P_t* \chi_{t,\epsilon}) (y-x) &\leq c_3   c_3 \Big( g_{t,\vartheta}(y-x)+ \rho_t h_\epsilon(|y-x|\rho_t)\Big).
\end{align*}
 Thus, under the assumptions of the theorem, we can rewrite the estimate obtained in Lemma~\ref{Phi-up} as
\begin{equation}
\big|\Phi_t(x,y)(x,y)|\leq C_1 t^{-1+\eta}\big( g_{t,\vartheta}(y-x) + h_{t,\epsilon}(y-x)\big), \label{LZ10}
\end{equation}
where $\eta=\epsilon/\sigma$.

\end{proof}

\subsection{Proof of Theorem~\ref{t-main1}}

By Lemma~\ref{Phi-up} and Lemma~\ref{gen}.a) we get
\begin{equation}
|\Psi_t(x,y)|\leq \sum_{k=1}^\infty \big|\Phi_t^{\star k} (x,y)\big| \leq C_0 t^{-1+\delta} \big(g_{t,\zeta}* \Pi_t\big)(y-x), \label{Phi-est}
\end{equation}
where  $\delta=\eta/2=\epsilon/(2\sigma)$,
\begin{equation}\label{Pi}
\Pi_t(du):=\sum_{k=1}^\infty A^k G_t^{(k)} (du),
\end{equation}
the family of probability measures $\{ G_t^{(k)},\, t>0,\, k\geq 1\}$ is given by \eqref{gk-new}, and $A\in(0,1)$ is some constant.

Since $G_t^{(k)}(\cdot)$, $k\geq 1$,  are the probability measures,  we have
\begin{equation}
\Pi_t(\rn)= A(1-A)^{-1}, \quad \quad t\in (0,1].  \label{pi2}
\end{equation}
 Thus, the series $\Psi_t(x,y)= \sum_{k=1}^\infty \Phi_t^{\star k} (x,y)$ converges for any $t>0$, $x,y\in \real$, uniformly on compact subsets.

Proceeding in the same way as above,  we get
\begin{equation}
\big|(p^0 \star \Psi)_t(x,y)\big|\leq  C_1 t^\delta \big(g_{t,\chi} * \tilde{\Pi}_t\big) (y-x), \label{up-11}
\end{equation}
for some  $\chi\in (0,\zeta)$, where
\begin{equation}\label{til-pi}
\tilde{\Pi}_t(dw)=(C_2\delta)^{-1} \int_0^1 \int_\rn (1-r)^{-1+\delta} \Pi_{t(1-r)}(dw-u)P_{tr}(du)dr
\end{equation}
is the probability measure for any $t\in [0,1]$; here $C_2=C_0 (1-A)A^{-1}$.
Thus, expression \eqref{sol10} is well-defined. \qed

\subsection{Proof of Theorem~\ref{t-jump3}}
By Lemma~\ref{Phi-up2} and Lemma~\ref{gen}.b) we get
\begin{equation}\label{Phi-est2}
|\Psi_t(x,y)|\leq \sum_{k=1}^\infty \big|\Phi_t^{\star k} (x,y)\big| \leq C t^{-1+\delta} \big(g_{t,\zeta}(y-x)+ h_{t,\epsilon}(y-x)\big),
\end{equation}
which together with the estimate on $p_t^0(x,y)$ given by Lemma~\ref{aux2} gives \eqref{txy2}.
\qed

 \section{Continuity and smoothness properties: Proof of Propositions~\ref{p-main1} and \ref{p-main2} }\label{cont}

\begin{proof}[Proof statements 1 and 2  of Proposition~\ref{p-main1}]
1.  Note that by \eqref{growth} and the theorem on continuity with respect to parameters the function  $p_t(x)$ is continuous in $(t,x)\subset (0,\infty)\times\real$. Note that we can rewrite  $(p^0\star \Psi)_t(x,y)$  as
  $$
  (p^0\star \Psi)_t(x,y)= \int_0^t \int_\real p_{t-s}^0(x,y+u)\Psi_s(u+y,y)duds.
  $$
  Recall that by Lemma~\ref{aux1} we have $|p_t^0(x,y)|\leq \big(g_{t} * P_t\big)(x-y)$. Then, using this estimate and \eqref{Phi-est} we derive
 for $0<s\leq t/2$  (cf.  the proof of Theorem~\ref{t-main1})
\begin{align*}
p_{t-s}^0 (x,u+y)\Psi(s,u+y,y)
&\leq \int_\real\int_\real g_{t-s} (y-x+u-w_1) g_{s,\zeta} (u-w_2) P_{t-s}(dw_1)\Pi_s(dw_2)
\\
%\leq c_1 \rho_t \int_\real\int_\real  f_{up}^{(1-\zeta)(1-\epsilon)}((y-x-w_1-w_2)\rho_t) \rho_s f_{up}^{(1-\zeta)\epsilon} ((u-w_2)\rho_s)P_{t-s}(dw_1)\Pi_s(dw_2)
%\\&
%\leq c_2(t_0) \int_\real \rho_s f_{up}^{(1-\zeta)\epsilon} ((u-w_2)\rho_s)\Pi_s(dw_2)
%\\&
&= c(t_0)\big( g_{s,\theta}* \Pi_s\big)(u),
\end{align*}
where  we used that $f_{up}(x)\leq d_1$, and  for $t\geq t_0>0$  the function $\rho_t$ is bounded by a constant, depending on $t_0$.  Here $\theta\in (0,\zeta)$. Analogous calculation for $t/2<s\leq t$ gives the same upper estimate.
Therefore,
\begin{equation}
p_{t-s}^0(x,u+y)\Psi_s(u+y,y)\leq c_3(t_0)  \big( g_{s,\theta}* \Pi_s\big)(u),
\end{equation}
with the right-hand side integrable on $[0,t]\times \real$. Thus, by the theorem on continuity with respect to parameters, $(p_t^0\star \Psi)_t(x,y)$ is jointly continuous in $(t,x,y)$ on $[t_0, \infty)\times \real\times \real$.

\medskip

2. The proof of  \eqref{eqII} is contained essentially in the proof of Theorem~\ref{t-main1}. Namely,
using representation \eqref{sol10}, the estimate for $p_t^0(x,y)$ which follows from
 Lemma~\ref{aux1}, and \eqref{up-11}, we get
\begin{equation}
\begin{split}
p_t(x,y) \leq  c_2  \big(g_{t,\chi} * Q_t\big) (x-y)
 \label{up1}
\end{split}
\end{equation}
for all $t\in (0,1]$, $x,y\in \real$,   $\chi\in(0,\zeta)$ (cf. \eqref{up-11}), and the probability measure
\begin{equation}\label{Q1}
Q_t(du)=  c\big( P_t(du)+ t^\delta (P\star \tilde{\Pi})_t(du)\big).
\end{equation}
Here $c>0$ is the normalizing constant, such that $Q_t(\real)=1$.

Note that in principle  in the procedure described above    $0<\chi<1$ can be chosen arbitrarily  close to 1.
 \end{proof}
For the proof of statements 3 and 4 of Proposition~\ref{p-main1} we need some auxiliary statements.
Let
\begin{equation}\label{mathp}
\mathcal{P}_t (dw):= P_t(dw)+ (P_t *\Lambda_t)(dw),
\end{equation}
where  $P_t(dw)$ is defined in \eqref{Poist}.
\begin{lemma}\label{p0-der}
The function $p_t^0(x,y)$ is differentiable with respect to $t$, the derivative $\partial_t p_t^0(x,y)$ is  continuous in  $(t,x,y)\in(0,\infty)\times \real\times \real$, and for all $k\geq 0$ we have
$$
\big|\partial_t \partial^k_x  p_t^0(x,y)\big|\leq C t^{-1}\rho_t^k \big(g_t * \mathcal{P}_t\big)(y-x), \quad t>0, \quad x,y\in \real,
$$
where   $\mathcal{P}_t(dw)$ is defined in \eqref{mathp},  $g_t$ is of the form \eqref{gtc} with some constant $c>0$.
\end{lemma}
The proof of this lemma can be obtained by modifying the proof of the upper estimate for $p_t^0(x,y)$ (cf. \eqref{ptx-der}), see  \cite{KK12a}. In order to make the paper self-contained, we give the proof in Appendix A.

This lemma allows to transfer the differentiability properties of $p_t^0(x,y)$ to $p_t(x,y)$. But for this we need to establish the  continuity  and upper estimates on $\partial_t \Phi^{\star k}$ and $\partial_t \Psi^{\star k}$, respectively.

\begin{lemma}\label{phi-der}
The function $\Psi_t(x,y)$ is differentiable with respect to $t$, $\partial_t \Psi_t(x,y)$ is  continuous in  $(t,x,y)\in(0,\infty)\times \real\times \real$, and there exists a family of measures $\{ \Theta_t, \, t\geq 0\}$,  such that
\begin{equation}\label{phi-der10}
\big|\partial_t \Psi_t(x,y)\big|\leq C t^{-1} (g_t* \Theta_t)(y-x), \quad t>0, \quad x,y\in \real.
\end{equation}
Here $\zeta\in (0,1)$ is some constant, $g_t$ is of the form \eqref{gtc} with some constant $c>0$.
\end{lemma}
\begin{proof}
We use the same argument as for the proof of Theorem~\ref{t-main1}. Using Lemma~\ref{p0-der} for $\partial_t \partial_x^2 p_t^0(x,y)$  one can obtain the estimate for $\partial_t \Phi_t(x,y)$ in the same way as it was done for $\Phi_t(x,y)$ in Lemma~\ref{Phi-up}:
\begin{equation} \label{der-Phi1}
\big|  \partial_t \Phi_t(x,y)\big|\leq C_1t^{-2+\delta}\big(g_{t} *\mathcal{G}_{t}\big)(y-x),
\end{equation}
where  $C_1>0$ is some constant,  and  the family of measures $\mathcal{G}_{t} (dw)$ is given by
\begin{equation}\label{G11}
\mathcal{G}_t(dw):=  \mathcal{P}_t(dw)+ (\mathcal{P}_t*
\chi_{t,\epsilon})(dw).
\end{equation}
Write
\begin{equation}\label{Phik}
\Phi^{\star (k+1)}_t(x,y)=\int_0^{t/2}\int_{\real}\Phi_{t-s}^{\star k}(x,z)\Phi_s(z,y)\,dz ds+\int_0^{t/2}\int_{\real}\Phi_{s}^{\star k}(x,z)\Phi_{t-s}(z,y)\,dz ds.
\end{equation}
It can be  shown  by induction  that  each $\Phi^{\star k}_t(x,y)$ has a continuous derivative with respect to $t$, and
\begin{equation}\label{Phik10}
\begin{split}
\prt_t\Phi^{\star (k+1)}_t(x,y)&=\int_0^{t/2}\int_{\real}(\prt_t\Phi^{\star k})_{t-s}(x,z)\Phi_s(z,y)\,dz ds+
\int_0^{t/2}\int_{\real}\Phi_{s}^{\star k}(x,z)(\prt_t\Phi)_{t-s}(z,y)\,dz ds\\&+\int_{\real}\Phi_{t/2}^{\star k}(x,z)\Phi_{t/2}(z,y)\, dz.
\end{split}
\end{equation}
By induction, we get
\begin{equation}\label{Phik20}
\big|   \partial_t \Phi_t^{\star k} (x,y)\big|\leq C_{k}    t^{-2+k\delta}
\big(g_{t}^{(k)} *\mathcal{G}_t^{(k)}\big)(y-x),\quad k\geq 2,
\end{equation}
where the sequence  $(g^{(k)}_t(x))_{k\geq 1}$ is the same as in Lemma~\ref{gen}, and
\begin{equation}\label{calgtk}
\begin{split}
\mathcal{G}^{(k)}_t(dw):&=\frac{1}{B((k-1)\delta,\delta)+1}\Big(  \int_0^1 \int_\rn r^{-1+\delta}(1-r)^{-1+\delta(k-1)} \mathcal{G}_{t(1-r)}^{(k-1)}(dw-u) \mathcal{G}_{tr} (du)dr \\
&\quad +\big( \mathcal{G}_{t/2}^{(k-1)} * \mathcal{G}_{t/2}\big)(dw)\Big), \quad k\geq 2.
\end{split}
\end{equation}
Take as before $k_0:= \big[ \tfrac{n}{\alpha \delta}\big] +1$. Then by the same argument as in Lemma~\ref{es-new} we get
\begin{equation}\label{Phik30}
\big| \partial_t \Phi_t^{\star (k_0+\ell)}(x,y) \big|\leq  D_\ell  t^{-2+ \delta (k_0+\ell)} \big(g_{t,\zeta} * \mathcal{G}_t^{(k_0+\ell)} \big) (y-x), \quad \ell\geq 1,
\end{equation}
where $D_\ell := C(k_0) K^\ell B((k_0+\ell-1) \delta, \delta)$, $\ell\geq 1$; here $C(k_0), K>0$ are some constants.

Finally, define
\begin{equation}
\Theta_{t}(du):=\sum_{k=1}^\infty A^k \mathcal{G}_{t}^{(k)}(du),  \label{Pit}
 \end{equation}
 where $A\in (0,1)$ is some  constant.
\begin{comment}\begin{equation}
U_{t} (dw):= \Lambda_t(dw) + \big(\Lambda_t \star \Theta_{t} \big)(dw), \label{Tett}
\end{equation}
and put
\begin{equation}\label{qu-til}
\tilde{Q}_t (dz):= \sum_{k=1}^\infty \frac{B^k}{k!} U^{*k} (dz),
\end{equation}
where $B>0$ is some constant.
\end{comment}
Then \eqref{phi-der10} follows from \eqref{Phik20} and \eqref{Phik30}.

\end{proof}

\begin{proof}[Proof of statements 3 and 4 of Proposition~\ref{p-main1}]
3.   The proof of differentiability of $p_t(x,y)$ essentially follows from Lemma~\ref{p0-der} and  Lemma~\ref{phi-der}. Indeed, writing $p_t(x,y)$ in the form
$$
p_t(x,y)
=p_t^0(x,y)+\int_0^{t/2}\int_{\real}p_{t-s}^0(x,z)\Psi_s(z,y)\,dzds+\int_{0}^{t/2}\int_{\real}p_{s}^0(x,z)\Psi_{t-s}(z,y)\,dz\, ds,
$$
and applying the above lemmas we get \eqref{der-es} with
$$
\tilde{Q}_t(dw):= c\big(\mathcal{P}_t(dw)+ (\mathcal{P}\star \Pi)_t(dw)+ (P\star \Theta)_t(dw)\big),
$$
where $\Pi_t(dw)$ is the measure appearing in \eqref{Phi-est}, the measure  $\Theta_t(dw)$ is given by \eqref{Pit}, and $c>0$ is the normalizing constant.

\end{proof}

\begin{proof}[Proof of Proposition~\ref{p-main2}]
1.  The proof of continuity of $T_tf$ follows by the same argument  as the proof of  the first statement of Proposition~\ref{p-main1}.
 To prove that $T_tf(x)$ vanishes as $|x|\to \infty$, observe that $p_t^0(x,y)\to 0, \quad |x|\to \infty$,  
\begin{align*}
\int_{\real} |(p^0\star \Psi)_t(x,y)| dy\leq C, \quad t\in (0, 1],
\end{align*}
 (see  \eqref{up-11}),  and for  every $t>0$
$$
\sup_x\int_{y:\,|y-x|>R}|(p^0\star \Psi)_t(x,y)|\, dy\to 0, \quad R\to\infty.
$$
Then statement 1  follows from the above relations.

\medskip

2.   Note that   for any $\phi\in B_b(\real)$ we have
\begin{align*}
\int_\real \phi(y) (p^0\star \Psi)_t(x,y)dy &\leq c_2t^\delta \int_\real \int_\real g_{t,\chi}(y-x-w) \big( P_t \star \Pi_t\big)(dw)dy
\\&
\leq
 c_3 t^{\delta}, \quad  t\in (0,1], \quad x\in \real.
\end{align*}
Since by the very definition of $p_t^0 (x,y)$ we have
$$
\sup_x \left| \int_{\real} p_t^0(x,y)\phi(y)dy-\phi(x)\right|\to 0, \quad t\to 0,
$$
we arrive at \eqref{Tt0}.

\medskip

3. Statement 3 follows from the respective statement 3 of Proposition~\ref{p-main1}.
\end{proof}

\section{Justification procedure. Proof of Theorem~\ref{t-main2}}\label{just}

Our approach follows the same line as that of the proof of  the justification presented in \cite{KK14a} and  the forthcoming paper \cite{KK14b}, in which  we extend this method to the case of a more general operator. Nevertheless, in order to make the paper self-contained and to simplify the reading, we give below the outline of this proof, skipping some easy but lengthy calculations.

The proof is based on the properties of the \emph{approximative}  fundamental solution.
Denote for $\epsilon>0$
\begin{equation}\label{pe}
p_{t,\epsilon}(x,y):=p_{t+\epsilon}^0(x,y)  + \int_0^t \int_{\real}  p_{t-s+\epsilon}^0(x,z)  \Psi_s(z,y) dzds.
\end{equation}
The function   $p_{t,\epsilon}(x,y)$ provides  a smooth approximation for $p_t(x,y)$ in the following sense.

 \begin{lemma}\label{aux-e}
 Let $p_{t,\epsilon}(x,y)$ be the function defined by \eqref{pe}. The statements below hold true.
 \begin{enumerate}

 \item For any $\epsilon>0$ the function  $p_{t,\epsilon}(x,y)$ is continuously  differentiable in $t$ and belongs to the class $C^2_\infty(\real)$ in $x$.

 \item  $p_{t,\epsilon}(x,y)\to p_t(x,y)$ as $\epsilon\to 0$, uniformly on compact sets in $(0,\infty)\times \real\times \real$.

 \item  For any  $\epsilon>0$ and $f\in C_\infty(\real)$,
  the function
  $ \int_{\real}p_{t,\epsilon}(x,y)f(y)\, dy$
  is continuously  differentiable in $t$ and belongs to the class $C^2_\infty(\real)$ w.r.t.  $x$.

\item For any   $f\in C_\infty(\real)$ and $t>0$,
   $$
  \lim_{\epsilon\to 0}  \int_{\real}p_{t,\epsilon}(x,y)f(y)\, dy=\int_\real  p_{t}(x,y)f(y)dy,
  $$
  uniformly w.r.t. $(t,x)\in [\tau,T]\times \real $ for any $\tau>0$, $T>\tau$.

\item For any   $f\in C_\infty(\real)$ and $t>0$,
$$
\lim_{\epsilon\to 0}  \prt_t\int_\real  p_{t,\eps}(x,y)f(y)dy=\int_\real \prt_t p_{t}(x,y)f(y)dy,
 $$
uniformly w.r.t. $(t,x)\in [\tau,T]\times \real $ for any $\tau>0$, $T>\tau$.

\end{enumerate}

\end{lemma}

   Denote
$$
q_{t,\epsilon}(x,y): = \Big( L(x,D)-\prt_t \Big) p_{t,\epsilon} (x,y).
$$
Observe  that $L(x,D)p_{t,\epsilon} (x,y)$ is well defined due to statement 1 in Lemma~\ref{aux-e}.

\begin{lemma}\label{l5}
For any $f\in C_\infty(\real)$  we have
\begin{itemize}
\item[(i)]
\begin{equation}\label{conv_loc}
\int_\real  q_{t,\eps}(x,y)f(y)dy\to 0, \quad \epsilon\to 0,
\end{equation}
uniformly w.r.t. $(t,x)\in [\tau,T]\times \real$ for any $\tau>0, T>\tau$, and
 \begin{equation}\label{conv_int}
\int_0^t \int_\real   q_{t,\eps}(x,y) f(y)dyds\to 0, \quad \epsilon\to 0,
\end{equation}
uniformly w.r.t. $(t,x)\in [0,T]\times \real$ for any $T>0$.
\item[(ii)]
$$
\int_\real p_{t,\epsilon}(x,y)f(y)dy\to f(x), \quad t,\epsilon\to 0,
$$
uniformly w.r.t $x\in \real$.
\end{itemize}
\end{lemma}

The proof of the above statements repeats the arguments of Lemma~5.1 in \cite{KK14a}, with the necessary modifications provided by the upper estimate for $p^0_t(x,y)$  (cf. \eqref{ptx-der}), the upper bound for $\Psi$ (cf.  \eqref{Phi-est}), their  derivatives (cf.  Lemmas~\ref{p0-der} and \ref{phi-der}),   and  Propositions~\ref{p-main1} and \ref{p-main2}.

Lemmas~\ref{aux-e} and \ref{l5} allow us to prove the following statement.

\begin{lemma}\label{pmp} The kernel $p_t(x,y)$ is  non-negative, possesses the  semigroup property, and for any $f\in {C}_\infty^2 (\real)$  one has
\begin{equation}\label{eq-norm}
\int_\real p_t(x,y)f(y)dy=f(x)+  \int_0^t \int_\real p_s (x,y)h_f(y)dyds, \quad t>0,
\end{equation}
where $h_f(x):=Lf(x)$,  which is well-defined for $f\in {C}_\infty^2 (\real)$.
\end{lemma}
\begin{proof}
 We show that $p_t(x,y)$ is non-negative; the proofs of the semigroup property and of \eqref{eq-norm} are analogous, and we refer to \cite{KK14a} for details.

Since $p_t(x,y)$ is continuous in $(t,x,y)$, it is enough to show that
\begin{equation}\label{int1}
\int_\real p_t(x,y)f(y)dy \geq 0
\end{equation}
for any $f\geq 0$, $f\in C_\infty(\real)$. Without loss of generality we assume that
\begin{equation}\label{f1}
\int_\real f(y)dy=1.
\end{equation}

Suppose that \eqref{int1} fails. Then there exist  $t_0>0$, $x_0\in \real$, and the function
$f$ as above such that for some $\theta>0$ we have
\begin{equation}\label{int2}
\int_\real p_{t_0}(x_0,y)f(y)dy <-\theta.
\end{equation}
By statement 4 in Lemma~\ref{aux-e}   we can approximate the integral in \eqref{int1} by $\int_\real p_{t,\epsilon}(x,y)f(y)dy$.
  Since   $f\geq 0$, then by statement  (ii) from Lemma \ref{l5} there exist $\tau_0>0$, $\eps_0>0$, such that
\begin{equation}\label{tau5}
\inf_{x\in \real, \tau\in(0, \tau_0], \eps\in (0, \eps_0]}\int_\real  p_{\tau, \eps}(x,y)f(y)dy>-\theta/3.
\end{equation}
Fix $\tau\in (0,\tau_0\wedge t_0)$ and $\mathfrak{T}\in (t_0, \infty)$.
By (\ref{int2}) and statement 4 in Lemma~\ref{aux-e} there exists $\epsilon_{\tau,\mathfrak{T}}>0$ such that
 \begin{equation}\label{inf}
\inf_{t\in [\tau,\mathfrak{T}],x\in \real} \int_\real p_{t,\epsilon} (x,y)f(y)dy \leq -\theta,\quad \eps\in (0, \eps_{\tau, \mathfrak{T}}).
  \end{equation}
Define the  function
\begin{equation}\label{ptilde}
\tilde{p}_{t,\epsilon} (x,y)= p_{t,\epsilon} (x,y)+ t \theta/(2\mathfrak{T}).
 \end{equation}
 By statements \eqref{inf} and \eqref{f1} we get
 $$
\inf_{t\in [\tau,\mathfrak{T}],x\in \real} \int_\real \tilde{p}_{t,\epsilon}  (x,y)f(y)dy  \leq -\theta/2<0,\quad \eps\in (0, \eps_{\tau, \mathfrak{T}}).
$$
  uniformly w.r.t. $t\in [\tau, \mathfrak{T}]$.  On the other hand, by statement 2 of Lemma~\ref{aux-e},
 \begin{equation}\label{lalala}
\lim_{|x|\to\infty}  \int_\real \tilde{p}_{t,\epsilon}  (x,y)f(y)dy \to t \theta/(2\mathfrak{T})>0
 \end{equation}
  uniformly w.r.t. $t\in [\tau, \mathfrak{T}]$. Thus,  for every  $\eps\in (0, \eps_{\tau, \mathfrak{T}})$ there exist $x_\eps\in \real, t_\eps\in [\tau, \mathfrak{T}]$, such that
\begin{equation}\label{glob_min}
\int_\real \tilde p_{t_\eps,\epsilon} (x_\eps,y)f(y)dy=\min_{t\in [\tau,\mathfrak{T}],x\in \real} \int_\real \tilde p_{t,\epsilon} (x,y)f(y)dy\leq -\theta/2<0.
\end{equation}
Observe, that since the convergence in \eqref{lalala} is uniform,  all  points $x_\eps$, $\eps\in (0,\eps_{\tau, \mathfrak{T}})$, belong to some compact set $K(\tau, \mathfrak{T},f)$, and by \eqref{tau5}  for $\epsilon\in (0,\epsilon_{\tau,\mathfrak{T}}\wedge \epsilon_0)$ we have
$$
  \int_\real p_{t_\eps,\epsilon} (x_\eps,y)f(y)dy\leq \int_\real \tilde p_{t_\eps,\epsilon} (x_\eps,y)f(y)dy\leq -\theta/2<-\theta/3,
  $$
implying  $t_\epsilon>\tau$.

 Take $\eps\in (0,\eps_{\tau, \mathfrak{T}}\wedge \eps_0)$; since the  minimum in (\ref{glob_min}) w.r.t. $(t,x)\in [\tau, \mathfrak{T}]\times \real$ is attained at some point $(t_\eps, x_\eps)\in (\tau, \mathfrak{T}]\in \real$, we conclude that
$$
\int_\real \prt_t \tilde{p}_{t,\epsilon}(x,y)f(y)dy|_{(t_\eps,x_\eps)} \leq  0
$$
(the inequality may appear if $t_\eps=\mathfrak{T}$), and since $L$ possesses the positive maximum principle,
$$
L_x\int_\real \tilde{p}_{t,\epsilon}(x,y)f(y)dy|_{(t_\eps,x_\eps)}= \int_\real L_x\tilde{p}_{t,\epsilon}(x,y)f(y)dy|_{(t_\eps,x_\eps)} \geq 0.
$$
Thus,
\begin{equation}\label{max1}
\Big( L_x -\prt_t \Big)  \int_\real \tilde{p}_{t,\epsilon}(x,y)f(y)dy\big|_{(t_\eps,x_\eps)}\geq 0.
\end{equation}
On the other hand,  by (\ref{conv_loc}) we have
\begin{equation}\label{max2}
 \Big(L_x-\prt_t \Big) \int_\real  \tilde{p}_{t,\epsilon}(x,y)f(y)dy=\int_\real  q_{t,\epsilon}(x,y)f(y)dy -\theta/(2\mathfrak{T})\to -\theta/(2\mathfrak{T}), \quad \eps\to 0,
\end{equation}
uniformly w.r.t. $t\in [\tau, \mathfrak{T}]$ and $x$ in any compact set $K$. Taking $K$ equal to $K(\tau, \mathfrak{T},f)$, which contains  all $x_\eps$ with small $\eps$,  we get a contradiction with (\ref{max1}).
\end{proof}

\begin{proof}[Proof of  statements I and II of Theorem \ref{t-main2}]

\emph{I.}   It follows from \eqref{eq-norm} that
\begin{equation}\label{id}
\int_\real p_t(x,y)dy=1, \quad t>0, \quad x\in \real.
\end{equation}
Indeed, take $f\in C^2_\infty(\real)$ such that $f\equiv 1$ on the unit ball in $\real$, and put $f_k(x)=f(k^{-1}x)$. Then \eqref{id} follows by the dominated convergence theorem.

Thus, by \eqref{id}, positivity of $p_t(x,y)$ and Proposition~\ref{p-main2}, $(T_t)_{t\geq 0}$ is  the positive strongly continuous contraction conservative semigroup on $C_\infty(\real)$, where by  $T_0 $ we understand the identity operator.  By continuity of the kernel $p_t(x,y)$, see Proposition~\ref{p-main1},  the respective Markov process $X$ is strong Feller.

\medskip

\emph{II.}
 Using the Markov property of $X$, it is easy to deduce from (\ref{eq-norm}) and the  semigroup property for $p_t(x,y)$ the following: for given $f\in C_\infty^2(\real)$, $t_2>t_1$, and $x\in\real$, for any $m\geq 1$, $r_1, \dots ,r_m\in [0, t_1]$, and bounded measurable $V:(\real)^m\to \real$ the identity
$$
\Ee^x\left[f(X_{t_2})-f(X_{t_2})-\int_{t_1}^{t_2}h_f(X_s)\, ds\right]V(X_{r_1}, \dots, X_{r_2})=0
$$
holds true. This means that for every $f\in C_\infty^2(\real)$ the process
\begin{equation}\label{Mf}
M^f_t=f(X_t)-\int_{0}^{t}h_f(X_s)\, ds, \quad t\geq 0
\end{equation}
is a $\Pp^x$-martingale for every $x\in \real$; that is, $X$ is a solution to the martingale problem (\ref{mart}).
\end{proof}

\subsection{Generator of the semigroup $(T_t)_{t\geq0}$. }

\begin{proof}[Proof of Statement III of Theorem~\ref{t-main2}]
   In the first step is to show  that   the generator $A$  of the semigroup $(T_t)_{t\geq 0}$ is well defined on ${C}_\infty^2 (\real)$, and its restriction to this space  coincides with $L$. The argument here is quite standard: first note that since for $\mathfrak{f}\in {C}_\infty^2 (\real)$ the process (\ref{Mf}) is a $\Pp^x$-martingale for every $x\in \real$, and $h_{\mathfrak{f}}$ is continuous, by Doob's optional sampling theorem the  Dynkin operator $\mathcal{U}$  (cf. \cite[Chapter II.5]{GS74}) is well defined on such $\mathfrak{f}$, and $\mathcal{U}\mathfrak{f}=L\mathfrak{f}$. Since $\mathcal{U}\mathfrak{f}$ is continuous, by \cite[Theorem II.5.1]{GS74}  we get that $\mathfrak{f}$ belongs to the domain of the generator $A$, and $A\mathfrak{f}=\mathcal{U}\mathfrak{f}=L\mathfrak{f}$.
Hence $(L, C_\infty^2 (\real))$ is a restriction of $(A, \mathcal{D}(A))$. Since $A$ is a closed operator, this yields that $(L, C_\infty^2 (\real))$ is closable. Let us show that its closure coincides with whole  $(A, \mathcal{D}(A))$.

Take $\mathfrak{f}\in C_\infty(\real)\cap \mathcal{D}(A)$.
Fix $t>0$, and consider the element $\mathfrak{f}_{t}=T_{t}\mathfrak{f}$. Write
$$
\mathfrak{f}_{t,\eps}(x)=\int_{\real}p_{t,\eps}(x,y)\mathfrak{f}(y)\, dy,
$$
and observe that we have the following properties.

\begin{itemize}
  \item By statement 3 in  Lemma \ref{aux-e}, $\mathfrak{f}_{t,\eps}\in C^2_\infty$.
  \item Since $A$ is an extension of $L$, then we have that $\mathfrak{f}_{t,\eps}\in D(A)$, and  $A\mathfrak{f}_{t, \eps}=L \mathfrak{f}_{t,\eps}$.
  \item  By statement 4 in  Lemma \ref{aux-e},   one has $\mathfrak{f}_{t,\eps}\to \mathfrak{f}_t$ in $C_\infty$ as $\eps\to 0$.
   \item  By  statement 4 in  Lemma \ref{aux-e},  one has  $\prt_t\mathfrak{f}_{t,\eps}\to \prt_t\mathfrak{f}_t$  in $C_\infty$ as $\eps\to 0$.
    \item By   Lemma \ref{l5},  one has $(\prt_t-L)\mathfrak{f}_{t,\eps}\to 0$ in $C_\infty$ as $\eps\to 0$.
\end{itemize}

Recall that $\mathfrak{f}\in \mathcal{D}(A)$, and therefore $\prt_t \mathfrak{f}_t=A\mathfrak{f}_t.$ Hence, summarizing all the above we get that  $\mathfrak{f}_{t,\eps}\in C^2_\infty(\real)$ approximates $\mathfrak{f}_t$, and $L\mathfrak{f}_{t,\eps}$ approximates $A\mathfrak{f}_t$ in $C_\infty(\real)$ when $\eps\to 0$. This gives that the domain of the $C_\infty(\real)$-closure of $(L, {C}_\infty^2 (\real))$ contains every element of the form
$\mathfrak{f}_t=T_t\mathfrak{f}$,   $t>0$, $\mathfrak{f}\in \mathcal{D}(A)$.
This clearly yields that this closure coincides with whole $(A, \mathcal{D}(A))$.

\end{proof}

\subsection{Proof of Proposition~\ref{lower}}
 By \eqref{up-11} we have
\begin{equation}
\big|(p^0\star \Psi)_t(x,y)\big|\leq c_1 \rho_t t^{\delta}, \quad x,y\in \real,\quad  t\in (0,1],\label{ZF1}
\end{equation}
which implies the upper bound $p_t(x,x) \leq c \rho_t$ for all $x\in \real$ and $t\in (0,1]$. To get the lower bound,  observe that by \eqref{ZF1}  and Lemma~\ref{aux1} for $t$ small enough  we have
\begin{align*}
p_t(x,x)&\geq p^0_t(x,y)-\big|(p^0\star \Psi)_t(x,y)|
\geq p^0_t(x,y) - c_1 \rho_t t^{\delta}
\geq c_2 \rho_t f_{low}(|y-x|\rho_t).
\end{align*}
In particular,  $p_t(x,x)\asymp \rho_t$ for all $x\in \real$,  $t\in (0,1]$. 
\qed

\section{Examples}\label{example}

\begin{example}\label{exa1}
 Consider a symmetric $\alpha$-stable process. In this case the associated  L\'evy measure $\mu$ is
 of the form $\mu(du)= c_\alpha|u|^{-1-\alpha}du$, $0<\alpha<2$, where $c_\alpha>0$ is some
appropriate constant, and the respective characteristic exponent $q(\xi)$ satisfies condition
 \textbf{A1} for any $0<\alpha<2$. Suppose that   conditions \textbf{A2} and \textbf{A3}
 are satisfied. By  Case II of Theorem~\ref{t-jump3} with $h(x)= 1\wedge |x|^{-1-\alpha}$ we have
\begin{equation}
\begin{split}
 p_t(x,y)&\leq c t^{-1/\alpha}\big( 1+ h(|x-y|t^{-1/\alpha})+ t^{\epsilon/\alpha} h_\epsilon(|x-y|t^{-1/\alpha})\big)
 \\&=
  \frac{c}{t^{1/\alpha}}\left( 1 + \Big\{\Big(\frac{t^{1/\alpha}}{|x-y|}\Big)^{1+\alpha}
+  t^{\epsilon/\alpha} \Big(\frac{t^{1/\alpha}}{|x-y|}\Big)^{1+\alpha-\epsilon}\Big\}\I_{|x-y|\geq t^{1/\alpha}}\right),\label{al-1}
 \end{split}
 \end{equation}
 for all $t\in (0,1]$, $x,y\in \real$.
\end{example}

\begin{example}\label{exa2}
Consider  a L\'evy process with the discrete L\'evy measure
$$
\mu(dy)=\sum_{k=-\infty}^\infty 2^{k\theta}\Big(\delta_{2^{-k\upsilon}}(dy)+\delta_{-2^{-k\upsilon}}(dy)\Big),
$$
where  $\upsilon>0$, and $0<\theta<2\upsilon$. This L\'evy measure was studied in detail in \cite{KK12a},
where we show that the respective characteristic exponent  satisfies $q(\xi) \asymp |\xi|^\alpha$ with
$\alpha=\theta/\upsilon$. In particular, condition \textbf{A1} is satisfied, and
$\rho_t\asymp t^{-1/\alpha}$.  One can check  that for $\alpha>1$ condition \eqref{dens02} holds true
with $1-\mathfrak{G}(x)= x^{-\alpha}\I_{x\geq 1}$, see \cite{KK12a}.  Suppose that conditions \textbf{A2} and
\textbf{A3} are satisfied.  Then,  by Case I of  Theorem~\ref{t-jump3} we have
\begin{equation}\label{exa-22}
\begin{split}
 p_t(x,y)&\leq  ct^{-1/\alpha}\big( 1+ h(|x-y|t^{-1/\alpha})+ t^{\epsilon/\alpha} h_\epsilon(|x-y|t^{-1/\alpha})\big)
 \\&
  =\frac{c}{t^{1/\alpha}}\left( 1 + \Big\{\Big(\frac{t^{1/\alpha}}{|x-y|}\Big)^{\alpha}
+ t^{\epsilon/\alpha}\Big(\frac{t^{1/\alpha}}{|x-y|}\Big)^{\alpha-\epsilon}\Big\}\I_{|x-y|\geq t^{1/\alpha}}\right),
 \end{split}
 \end{equation}
 for all $t\in (0,1]$, $x,y\in \real$.
\end{example}

\begin{example}\label{exa3}
Consider the function $\alpha: [0,\infty)\to [\alpha_-,\alpha_+]\subset (0,2)$ such that
    $    v\alpha'(v)\to 0$ as $v\to \infty$. 
It was shown in \cite{KK12a} that there exist a L\'evy measure $\mu$ and the respective characteristic exponent $q(\xi)$, such that
    \begin{equation}
    q^U(\xi)\asymp q(\xi)\asymp q^L(\xi)\asymp |\xi|^{\alpha(\ln |\xi|)}, \quad |\xi|\to\infty.\label{exa-31}
    \end{equation}
In particular, condition \textbf{A1} is satisfied. On the other hand, for any $\alpha\in [\alpha_-,\alpha_+]$   there exists a sequence $\{\xi_{\alpha,k}$, $k\geq 1\}$, such that
    $$
    \xi_{\alpha,k}\to \infty, \quad q(\xi_{\alpha,k})\asymp (\xi_{\alpha,k})^\alpha, \quad k\to\infty.
    $$
This example illustrates that despite of the fact that the characteristic exponent has oscillations, there exists the fundamental solution to \eqref{fund1}, provided that the function $m(x,u)$, which is responsible for the perturbations of $\mu$, satisfies the conditions of  Theorem ~\ref{t-main1}, which is the transition probability density of a Feller Markov process.
\end{example}

\section{Appendix A}

\begin{proof}[Proof of Lemma~\ref{p0-der}]  We follow the ideas presented in \cite[Section 3.2]{KK12a}.

By \textbf{A1}, we can bring the derivatives $\partial_t \partial^k_x $ inside the integral in the representation of $p_t(x)$. Split
\begin{align*}
\partial_t\partial_x^k  p_t(x) &= - \int_\real (-i\xi)^k q(\xi)e^{-i\xi x-tq(\xi)}d\xi\\
&= -  \int_\real (-i\xi)^k q_t(\xi)e^{-i\xi x-tq(\xi)}d\xi -  \int_\real (-i\xi)^k \big(q(\xi)-q_t(\xi)\big)e^{-i\xi x-tq(\xi)}d\xi\\
&= -\big(\partial^k \tilde{p}_t* P_t\big)(x)-t^{-1}\big( \partial^k \overline{p}_t* P_t* \Lambda_t\big)(x)\\
&=I_1(t,x)+I_2(t,x),
\end{align*}
where
$$
q_t(\xi):= \int_{|u\rho_t|\leq 1} (1-\cos(\xi u))\mu(du),
$$
$$
\tilde{p}_t(x):= \mathcal{F}^{-1} \big( q_t (\cdot) e^{-tq_t(\cdot)}\big)(x), \quad \overline{p}_t (x): =  \mathcal{F}^{-1} \big(  e^{-tq_t(\cdot)}\big)(x),
$$
and
$$P_t(dw)=\mathcal{F}^{-1}\Big(e^{-t(q(\cdot)- q_t(\cdot))}\Big)(dw),
$$
which coincides with the definition of $P_t(dw)$ given in  \eqref{Poist}.
The estimate for $\overline{p}_{t,k} (x)$ was obtained in \cite{KK12a}:
$$
\partial^k_x \overline{p}_t(x)\leq \rho_t^{k} g_t(x),
$$
where $g_t$ is of the from \eqref{gtc} with some constant $c>0$.
Therefore,
$$
I_2(t,x)\leq  t^{-1} \rho_t^k \big(g_{t}* P_t* \Lambda_t\big)(x), \quad t\in (0,1], \quad x\in \real.
$$
where
Note that $\big(\Lambda_t * P_t\big)(\real)\leq C $ for all $t\in [0,1]$.

Let us estimate $\tilde{p}_t(x)$.
Note that the function $q_t(\xi)$ can be extended to the complex plane with respect to $\xi$, and
$$
q_t(y+ i \eta)= \int_{|u\rho_t|\leq 1} (1-\cosh (\eta u) \cos (y u))\mu(du).
$$
Applying the Cauchy theorem, we have (see \cite{KK12a} for details)
\begin{align*}
I_1(t,x)= (2\pi)^{-1} \int_\real  (-iy+\eta)^k q_t(y+i\eta) e^{\eta x-i xy -tq_t (y+i\eta)}dy.
\end{align*}
Observe also, that $q_t(y+i\eta)$ is real-valued, and $q_t(y+ i \eta)\geq q_t(y)+q_t(i\eta)$.
Note that
$$
tq_t(y) \geq t  q(y) -t \int_{|u\rho_t|>1} (1-\cos (y u))\mu(du)\geq tq(y)- c_1.
$$
Since the L\'evy measure $\mu$ is symmetric, we have
\begin{align*}
-q_t(i\eta)&= t\int_{|\rho_t u|<1}[\cosh(\eta u)-1]\mu(du)=t\int_{|\rho_t u|<1}(\eta u)^2\vartheta(\eta u)\mu(du)
\\&
\leq t\vartheta(\eta/\rho_t)\int_{|\rho_t u|<1}(\eta u)^2\mu(du)
=t(\eta/\rho_t)^2
\vartheta(\eta/\rho_t)q^L(\rho_t)
\\&
=\cosh(\eta/\rho_t)-1,
\end{align*}
where $\vartheta(x)=x^{-2}[\cosh x-1]$, and we used that $\vartheta$ is even and strictly increasing on $(0,\infty)$.
 In such a way,
\begin{align*}
I_1(t,x)&\leq c_2  e^{\eta x+ \cosh (\eta /\rho_t)} \cosh (\eta/\rho_t) \int_\real (|\eta|+|y|)^kq_t (y)  e^{-tq(y)} dy\\
&\leq 2c_2  t^{-1} e^{\eta x+ 2\cosh (\eta /\rho_t)} \int_\real (|\eta|+|y|)^k  e^{-2^{-1}tq(y)} dy,
\end{align*}
where in the last line we used the obvious inequality $q_t(\xi)\leq q(\xi)$. The expression in the last line can be estimated from above in the same way as in \cite[Lemma 3.6]{KK12a}:
$$
I_1(t,x)\leq c_2 t^{-1} \rho_t^k \big(g_{t,\theta} * P_t\big)(x),
$$
where $g_t$ is the same as above in the proof, and $\theta\in (0,1)$.
Summarizing the estimates for $I_1(t,x)$ and $I_2(t,x)$, we derive the statement of Lemma~\ref{p0-der}.
\end{proof}

\section{Appendix B}

In this Appendix we formulate the result on the existence of the fundamental solution to \eqref{fund1} and its properties, under the assumption that the operator $L(x,D)$ is the sum of a L\'evy generator and a \emph{bounded} perturbation.

Let $\mu_\epsilon(du):= (|u|^\epsilon\wedge 1)\mu(du)$,
$$
\mathcal{R}^{(1)}_t(du)= P_t(du)+(P_t* \mu_\epsilon)(du),
$$
\begin{equation}\label{uk-new}
\mathcal{R}^{(k+1)}_t(dw)=\tfrac{1}{ B(k(1-\delta),1-\delta)}\int_0^1 \int_\real (1-r)^{k-1-k\delta} r^{-\delta}  \mathcal{R}_{t(1-r)}^{(k)}(dw-u) \mathcal{R}^{(1)}_{tr}(du)dr, \quad k\geq 1.
\end{equation}
For $A>0$ put
\begin{equation}\label{Q1-til}
\mathcal{R}_{t}(dw):= \sum_{k=1}^\infty A^k \mathcal{R}_{t}^{(k)}(dw),\quad 
\mathcal{Q}_{t} (dw):=c\big(P_t(dw)+ \big(P_t \star \mathcal{R}_{t}\big)(dw)\big),
\end{equation}
where $c>0$ is the normalizing constant.

\begin{theorem}\label{t-jump2}
Let the operator $L(x,D)$ be given by \eqref{lxd0}. Assume that conditions \textbf{A1}--\textbf{A3} are satisfied, but
 $\epsilon>0$ in condition \textbf{A3} is such that \eqref{inf1} fails. Then statements of Theorems~\ref{t-main1}, \ref{t-main2}, Propositions~\ref{p-main1}, \ref{p-main2} and \ref{lower} remain valid, with  the only modification:
$$
p_t(x,y)\leq \rho_t \Big(f_{up}(\rho_t \cdot ) * \mathcal{Q}_t\Big)(y-x), \quad t\in (0,1], \quad x,y\in \real.
$$
where the probability measure $\mathcal{Q}_t(dw)$ is given by \eqref{Q1-til}.
\end{theorem}
%\begin{remark}
%Clearly, under the assumptions of Theorem~\ref{t-jump2},  the result stated in Theorem~\ref{loc-time} also holds true.
%\end{remark}
\begin{proof}   The only difference of the proof from those  of Theorem~\ref{t-main1}  is in the estimate obtained in Lemma~\ref{Phi-up}. In this case for $\epsilon>0$ such that \eqref{inf1} fails  we get
 \begin{align*}
 \big|\Phi_t(x,y)\big|&\leq C \int_\real  \big(p_{t,y}(x+u)+p_t(x)\big)(|u|^\epsilon\wedge 1)  \mu(du)
 \\&
 =C\rho_t \Big(f_{up}(\cdot \rho_t) *\big( P_t+ P_t * \mu_\epsilon\big)\Big) (x)\\
 &= C  t^{-\delta} \big( \tilde{g}_t * \mathcal{R}^{(1)}_t)(y-x).
 \end{align*}
where as before $\tilde{g}_t(x)=t^\delta g_t(x)$, and $\delta\in (0,1)$ is arbitrary. The rest of the proof can be conducted in the same way as the the case of Theorem~\ref{t-main1}.
\end{proof}

\section{Appendix C}

\begin{proof}[Proof of  Lemma~\ref{con1}]
i)    Note that for $s<\frac{t}{2}$ we have $\rho_{t-s}\leq \rho_{t/2}$, implying
 \begin{align*}
    \int_\real
   \rho_{t-s}\rho_s f_{up}((x-z) \rho_{t-s}) f_{up} ((z-y)\rho_s)dz
    &
    \leq \rho_{t/2} \int_\real \rho_s  f_{up}((z-y)\rho_s)dz
    \leq  c_1 \rho_{t/2}
    \leq c_2\rho_{t},
    \end{align*}
where in the last inequality we used that $\rho_t\asymp \rho_{ct}$ for any $c>0$, $t\in (0,1]$, which is implied by  $q^U(\xi)\leq q^U(c\xi)\leq c^2 q^U(\xi)$ for any $c\geq 1$, $\xi\in\real$.  Analogously, for $s>\frac{t}{2}$
    \begin{align*}
   & \int_\real
    \rho_{t-s}\rho_s f_{up}((x-z) \rho_{t-s}) f_{up}((z-y)\rho_s)dz
    \leq c_3\rho_t.
    \end{align*}
Let $D(x):= |x|\ln (1+|x|)$. Since  $\rho_t$ is monotone increasing as $t\to 0$, we have by convexity of $D$ the inequalities
    \begin{align*}
   D(|x-z|\rho_{t-s}) +D(|z-y|\rho_s) &\geq 2D(2^{-1}|x-y| \rho_t)
   \geq D(|x-y|\rho_t) -c,
    \end{align*}
where $c>0$ is some constant. Then for any $\vartheta  \in (0,1)$ we  have
    \begin{equation}\label{g2}
    \begin{split}
    I(t,x,y)&\leq  d_1^2  e^{-d_2\vartheta \{ D(|x-y|\rho_t)-c\}}
        \\&
    \quad \cdot \int_\real
   \rho_{t-s} \rho_s \exp\Big[-(1-\vartheta)d_2 \big\{D(|x-z|\rho_{t-s}) + D(|z-y|\rho_s)\big\} \Big]dz
     \\&
     \leq C (1-\vartheta)^{-1}  \rho_t f_{up}^{\vartheta} ((y-x)\rho_t)
     =C (1-\vartheta)^{-1}  g_{t,\vartheta}(y-x),
    \end{split}
    \end{equation}
where   $C>0$ is some constant.

ii)  By assumption a) of Theorem~\ref{t-jump3}.II we have $h(Cx)\leq  h(x)$ for any $C,x\geq 1$, which implies
\begin{equation}
\begin{cases}
h_\epsilon (|x-z|\rho_{t-s})\leq h_\epsilon(|x|\rho_{t-s}/2), & \text{ when $|x-z|\geq |x|/2$,}
\\
h_\epsilon (|z|\rho_s)\leq h_\epsilon(|x|\rho_s/2), & \text{when $|x-z|\leq |x|/2$}.
\end{cases}
\end{equation}
where for the second relation we used that $|x-z|\leq |x|/2$ implies $|x|/2\leq |z|$. Take $B_1:= \{ z:\,\,|x-z|\geq |x|/2\}$, $B_2:= \{z:\,\,|x-z|\leq |x|/2\}$. We have
\begin{align*}
&\int_{B_1} h_{t-s,\epsilon}(x-z) h_{s,\epsilon}(z)\, dz \leq h_{t-s,\epsilon}(x/2) \int_\real  h_{s,\epsilon} (z)dz
\leq c_1 h_{t,\epsilon} (x/2)
\leq c_2 h_{t,\epsilon} (x), \\
&\int_{B_2} h_{t-s,\epsilon}(x-z) h_{s,\epsilon}(z)\, dz \leq   h_{s,\epsilon}(x/2)  \int_\real h_{t-s,\epsilon}(x-z) dz
\leq  c_1h_{t,\epsilon}(x/2)
\leq c_2 h_{t,\epsilon}(x),
\end{align*}
where in the last inequalities we used subsequently that $\rho_s$ increases monotone as $s\to 0$, $h_\epsilon(Cx)\leq C^{-1}  h_\epsilon (x)$ for any $C,\,x\geq 1$, and $h_\epsilon(x/2)\leq c h_\epsilon(x)$ for all $x\geq 1$. Thus, summarizing the above estimates,   we obtain \eqref{h-es}.
\end{proof}
\textbf{Acknowledgement.}  The authors thank N. Jacob and  R. Schilling for inspiring remarks and suggestions, and gratefully acknowledge the DFG Grant Schi~419/8-1.  The first-names author gratefully acknowledges  the Scholarship of the President of Ukraine for young scientists (2012-2014).

\begin{comment}
 \begin{equation}
    \begin{split}
    I(t,x,y)&\leq  d_3^2  e^{-d_4\varsigma\theta_{k-1}\{ D(|x-y|\rho_t)-c\}}
        \\&
    \quad \cdot \int_\real
   \rho_{t-s} \rho_s \exp\Big[-(1-\varsigma)\theta_{k-1} \big\{D(|x-z|\rho_{t-s}) + D(|z-y|\rho_s)\big\} \Big]dz
     \\&
     \leq C\theta_{k-1}^{-1} \rho_t f_{up}^{\theta_k} ((y-x)\rho_t)
     \\&
     =C\theta_{k-1}^{-1}  g_{t,\theta_k}(y-x),\label{g2}
    \end{split}
    \end{equation}
    \end{comment}

\end{document}